\documentclass[a4paper,11pt]{article}
\usepackage{cite}
\usepackage{amsmath}
\usepackage{amsthm}
\usepackage{amsfonts}
\usepackage{mathtools}
\usepackage{dsfont}
\usepackage{microtype}
\usepackage{color}
\usepackage{xspace}
\usepackage{import}
\usepackage{pdfpages}
\usepackage{fullpage}

\usepackage{cleveref}
\crefname{section}{Section}{Sections}
\Crefname{section}{Section}{Sections}
\crefname{equation}{}{}
\Crefname{equation}{}{}
\crefname{proposition}{Proposition}{Propositions}
\crefname{lemma}{Lemma}{Lemmas}
\crefname{theorem}{Theorem}{Theorems}
\crefname{corollary}{Corollary}{Corollaries}
\Crefname{corollary}{Corollary}{Corollaries}
\crefname{definition}{Definition}{Definitions}
\Crefname{definition}{Definition}{Definitions}

\theoremstyle{plain}
\newtheorem{theorem}{Theorem}[section]
\newtheorem{corollary}[theorem]{Corollary}
\newtheorem{proposition}[theorem]{Proposition}
\newtheorem{lemma}[theorem]{Lemma}

\theoremstyle{definition}
\newtheorem{definition}[theorem]{Definition}
\newtheorem{note}[theorem]{Note}
\newtheorem{remark}[theorem]{Remark}

\numberwithin{equation}{section}

\newcommand{\mb}[1]{\mathbb{#1}}
\newcommand{\mc}[1]{\mathcal{#1}}

\newcommand{\ie}{i.e.~}
\newcommand{\grad}{\nabla}
\newcommand{\ind}{\mathds{1}}
\newcommand{\cov}{\mathrm{cov}}

\renewcommand{\to}{\downarrow}

\newcommand{\bracket}[1]{\left[ #1 \right]}
\newcommand{\p}[1]{\mathbb{P}\bracket{#1}}
\newcommand{\e}[1]{\mathbb{E}\bracket{#1}}

\newcommand{\ipd}[2]{ \left< #1, \, #2 \right> }
\newcommand{\dipd}[2]{ \ipd{#1}{#2}_\grad }

\title{Liouville Brownian Motion \\ and Thick Points of the Gaussian Free Field}
\author{Henry Jackson\footnote{Statistical Laboratory, University of Cambridge. Research supported by EPSRC grant EP/H023348/1 for the Cambridge Centre for Analysis}}
\date{\today}

\begin{document}
\maketitle

\begin{abstract}
	We find a lower bound for the Hausdorff dimension that a Liouville Brownian motion spends in $\alpha$-thick points of the Gaussian Free Field, where $\alpha$ is not necessarily equal to the parameter used in the construction of the geometry. This completes a conjecture in \cite{berestycki2013diffusion}, where the corresponding upper bound was shown. 

	In the course of the proof, we obtain estimates on the (Euclidean) diffusivity exponent, which depends strongly on the nature of the starting point. For a Liouville typical point, it is $1/(2 - \frac{\gamma^2}{2})$. In particular, for $\gamma > \sqrt{2}$, the path is Lebesgue-almost everywhere differentiable, almost surely. This provides a detailed description of the multifractal nature of Liouville Brownian motion. 
\end{abstract}

\noindent\textbf{Key words or phrases:} Liouville quantum gravity, Liouville Brownian motion, Gaussian multiplicative chaos.


\section{Introduction}\label{section:intro}

The goal of this paper is to study the multifractal nature of Liouville Brownian motion. This is a process which was introduced in \cite{berestycki2013diffusion} and \cite{garban2013liouville} as the canonical diffusion in planar Liouville quantum gravity. For instance, it is the conjectured scaling limit of a simple random walk on a uniform random triangulation, conformally embedded into the plane (via circle packing, for example). Liouville quantum gravity and its geometry has itself been at the centre of remarkable developments. We point out, among many other works, \cite{rhodes2011kpz,duplantier2011liouville,berestycki2014kpz}. 

Liouville Brownian motion is a useful tool for studying the geometry of Liouville quantum gravity. In fact, Watabiki has already considered the object (in a non rigorous way) in an attempt to describe the metric and fractal structure of Liouville quantum gravity \cite{watabiki1993analytic}. This led him to propose a formula for the Hausdorff dimension of the random metric space. The paper \cite{berestycki2014kpz} may be viewed as a first rigorous step in studying such multifractal aspects using Liouville Brownian motion. The current paper addresses a similar point, but from a different perspective. 

The general structure of the paper is as follows. In the remainder of this section, we state the main results and try to give the intuitive idea behind the proof. In \cref{section:setup} we will briefly introduce the objects and definitions we use throughout the paper, providing references for the reader should they need more detail. In \cref{subs:moments} we show that the time change function has finite moments around times when the Brownian motion is conditioned to be in a thick point, and derive crude tail estimates for the time change process from those bounds. \cref{subsec:scaling} is spent proving a simple large-deviation type result for the supremum of the harmonic projection of the Gaussian Free Field on to a disc, to use as an analogue of the scaling relation enjoyed by exactly stochastically scale invariant fields. In \cref{subsec:holder} we combine the results from the previous sections to show H\"older like properties of the time change function $F_\gamma$. Finally, in \cref{subs:proofs,subs:regularity}, we prove the main theorems, using the regularity results obtained in \cref{subsec:holder}.

\subsection{Statement of results}\label{subs:statement}

Let $h$ be a zero boundary Gaussian Free Field, defined in a simply connected proper domain $\mc D \subset \mb C$. One of the difficulties of working with a GFF is that it is not defined as a function, so we cannot say what value $h(z)$ takes, for $z \in \mc D$. However, it is regular enough that we can talk about its average value on a set. We will usually take that set to be the circle of radius $\varepsilon >0$ centred at a point $z \in \mc D$, and call that average $h_\varepsilon(z)$. Let $\left\{ h_\varepsilon(z) \; ; \; \varepsilon > 0 \right\}_{z \in \mc D}$ be the circle averages of $h$. We will define both the GFF and its averages more precisely in \cref{subs:gff}. For $\alpha > 0$, the set $\mc T_\alpha$ of $\alpha$-thick points is given by
\begin{equation*}
	\mc T_\alpha = \left\{ z \in \mc D \, : \, \lim_{\varepsilon \to 0}\frac{h_\varepsilon(z)}{\log \frac{1}{\varepsilon}} = \alpha \right\}.
\end{equation*}
By a theorem in \cite{hu2010thick}, it is known that the Hausdorff dimension of $\mc T_\alpha$ is
\begin{equation*}
	\dim_H(\mc T_\alpha) = \max\left(0, 2 - \frac{\alpha^2}{2}\right)
\end{equation*}
almost surely. 

Let $0 < \gamma < 2$. We will denote by $Z^\gamma$ a $\gamma$-Liouville Brownian motion, formally defined as follows. Let $B$ be a Brownian motion killed upon leaving $\mc D$. We define its clock process, $F_\gamma$, to be
\begin{equation*}
	F_\gamma(t) = \int_0^t e^{\gamma h(B_s) - \frac{\gamma^2}{2} \mb E [h(B_s)^2]}ds,
\end{equation*}
and the LBM is given by $Z^\gamma_t = B_{F^{-1}_\gamma(t)}$. It is not trivial to make sense of this definition. This was done in \cite{berestycki2013diffusion} and \cite{garban2013liouville}, where further properties were also proved. We recall the construction more precisely in \cref{subs:lbm}. 

The main goal of this paper is to prove the following bound:

\begin{theorem}
	Let $\alpha, \gamma \in [0,2)$, and let $Z^\gamma$ denote a $\gamma$-Liouville Brownian motion. Then
	\begin{equation*}
		\dim_H(\left\{ t \, : \, Z^\gamma_t \in \mc T_\alpha \right\}) \geq \frac{1-\frac{\alpha^2}{4}}{1 - \frac{\alpha\gamma}{2}+\frac{\gamma^2}{2}},
	\end{equation*}
	almost surely, where $\dim_H$ refers to the Hausdorff dimension.
	\label{thm:lower.bound}
\end{theorem}

The proof of \cref{thm:lower.bound} follows similar lines as the proof of Theorem 4.1 in \cite{rhodes2013gaussian}. Long range correlations introduced by the Brownian motion created a few more technical difficulties to overcome. The authors proved their result for an exactly stochastically scale invariant field, and claimed that the result generalised from that to all log-correlated Gaussian fields. We were also unable to follow the generalisation of their proof, which used a non-trivial application of Kahane's convexity inequality -- we instead had to rely on \cref{lem:gff.scaling}.

\cref{thm:lower.bound}, combined with Theorem 1.4 in \cite{berestycki2013diffusion}, gives us the following corollary:

\begin{corollary}
	Let $\alpha, \gamma \in [0,2)$, and let $Z^\gamma$ denote a $\gamma$-Liouville Brownian motion. Then
	\begin{equation*}
		\dim_H(\left\{ t \, : \, Z^\gamma_t \in \mc T_\alpha \right\}) = \frac{1-\frac{\alpha^2}{4}}{1 - \frac{\alpha\gamma}{2}+\frac{\gamma^2}{2}},
	\end{equation*}
	almost surely.
	\label{thm:upper.bound}
\end{corollary}
Similar results for diffusions on deterministic fractals were given in \cite{hambly2003diffusion}.

The key to our proof is good estimates on the regularity of the time change $F_\gamma$ around $\alpha$-thick points. For a given $\alpha$, we do not get the regularity results around all of the $\alpha$-thick points, but we do get it around almost all of them, for the correct choice of measure. If we define the measure $\mu_\alpha$ by setting, for $s \leq t$,
\begin{equation*}
	\mu_\alpha([s,t]) = F_\alpha(t) - F_\alpha(s),
\end{equation*}
then we will show that $F_\gamma$ behaves polynomially for $\mu_\alpha$-almost every $t$, in the following sense:
\begin{theorem}
	For $\mu_\alpha$-almost every $t > 0$, the change of time $F_\gamma$ has the following growth rate:
	\begin{equation*}
		\lim_{r \rightarrow 0} \frac{\log |F_\gamma(t) - F_\gamma(t+r)|}{\log |r|} = 1 - \frac{\alpha\gamma}{2} + \frac{\gamma^2}{4},
	\end{equation*}
	almost surely. 
	\label{thm:regularity}
\end{theorem}

When we combine the regularity of the time change function $F_\gamma$ with known regularity properties of Brownian motion, we are able to find a bound on the small time behaviour of the LBM. Let us call $M_\alpha$ the Liouville measure constructed from a GFF with parameter $\alpha$, which is formally defined as
\begin{equation*}
	M_\alpha(dz) = e^{\alpha h(z)}dz.
\end{equation*}
The measure $M_\alpha$ is almost surely supported on the set of $\alpha$-thick points and so, if we choose a point in $\mc D$ according to $M_\alpha$, it will almost surely be an $\alpha$-thick point of the GFF. 

\begin{corollary}
	Suppose that the starting point of a $\gamma$-Liouville Brownian motion is chosen according to $M_\alpha$, i.e.~$Z^\gamma_0 \sim M_\alpha$. Then
	\begin{equation*}
		\limsup_{t \to 0}\frac{\log |Z^\gamma_t|}{\log t} = \frac{1}{2 - \alpha\gamma+ \frac{\gamma^2}{2}}
	\end{equation*}
	almost surely.
	\label{cor:regularity}
\end{corollary}

\begin{remark}
	When $\alpha = \gamma$ (which will be the typical case), the diffusivity exponent is $2 - \frac{\gamma^2}{2}$. Also observe that a single process can be both superdiffusive (e.g.~when $\alpha = 0$) and subdiffusive (e.g.~when $\alpha = \gamma$). 
\end{remark}

Finally, we will show the following result about the differentiability of a Liouville Brownian motion, for certain values of the parameter $\gamma$. 
\begin{corollary}
	Let $\gamma \in (\sqrt 2, 2)$. Then the $\gamma$-Liouville Brownian motion $Z^\gamma$ is Lebesgue-almost everywhere differentiable with derivative zero, almost surely. 
	\label{cor:differentiability}
\end{corollary}

\subsection{Intuition behind the proof}\label{subs:intuition}

Since we are looking at the dimension of times that a $\gamma$-LBM spent in $\alpha$-thick points, it helps us to first note the following lemma. It is not used in the proofs of the main theorems, but it provides motivation for them. 

\begin{lemma}
	Let $\alpha \in [0,2)$.
	The Hausdorff dimension of time that a Brownian motion $B$ spends in the $\alpha$-thick points of a GFF is given by
	\begin{equation*}
		\dim_H\left( \left\{ t \; : \; B_t \in \mc T_\alpha \right\} \right) = 1 - \frac{\alpha^2}{4}
	\end{equation*}
	almost surely. 
	\label{lem:kaufmann}
\end{lemma}

\begin{proof}
	Let $[B]$ denote the path of the Brownian motion $B$. Kauffman's dimension doubling formula for Brownian motion (see, for example, Theorem 9.28 of \cite{morters2010brownian}), tells us that
	\begin{equation*}
		\dim_H\left( \left\{ t \; : \; B_t \in \mc T_\alpha \right\} \right) = 2 \dim_H(\mc T_\alpha \cap [B])
	\end{equation*}
	almost surely. But then, since $B$ is independent of the GFF and hence $\mc T_\alpha$, by a theorem due to Hawkes \cite{hawkes1971hausdorff,hawkes1971some} (clearly stated and proved as Corollary 5.2 in \cite{peres1996intersection}), we know that
	\begin{equation*}
		\dim_H(\mc T_\alpha \cap [B]) = \dim_H(\mc T_\alpha)
	\end{equation*}
	almost surely. The result in \cite{hu2010thick} gives us that
	\begin{equation*}
		\dim_H(\mc T_\alpha) = 2 - \frac{\alpha^2}{2}
	\end{equation*}
	almost surely, which completes our proof. 
\end{proof}

Recall the result that, if a function $f$ is $\beta$-H\"older continuous, then for any suitable set $E$ we have the bound
\begin{equation}
	\dim_H(f(E)) \leq \frac{1}{\beta} \dim_H(E).
	\label{eq:holder}
\end{equation}

Let us call the set $T_\alpha = \left\{ t \: : \; B_t \in \mc T_\alpha \right\}$. Then notice that $F_\gamma(T_\alpha)$ is the set of time spent by $\gamma$-Liouville Brownian motion in $\alpha$-thick points. We will show that the inverse of the change of time, $F_\gamma^{-1}$ is $\frac{1}{1 - \frac{\alpha\gamma}{2} + \frac{\gamma^2}{4}}$-H\"older continuous around $\alpha$ thick points, allowing us to see that
\begin{equation*}
	1 - \frac{\alpha^2}{4} =  \dim_H(T_\alpha) = \dim_H(F_\gamma^{-1}(F_\gamma(T_\alpha))) \leq \left( 1 - \frac{\alpha\gamma}{2} + \frac{\gamma^2}{4} \right) \dim_H(F_\gamma(T_\alpha)),
\end{equation*}
where the final inequality comes from a result very similar to that in \cref{eq:holder}. (We cannot use that result exactly, since the H\"older continuity property of $F_\gamma^{-1}$ is restricted to a subset of its domain. We will discuss this further in \cref{subs:hd}.)

Rather than showing the regularity of $F_\gamma^{-1}$ directly, we will find properties of $F_\gamma$ and use them to deduce results about the inverse. However, showing the regularity properties of $F_\gamma$ around a single thick point, while useful, is not enough. We want to look at the regularity of $F_\gamma$ simultaneously around all $\alpha$-thick points that the Brownian motion $B$ visits. This is where we use the upper bound that was previously found in \cite{berestycki2013diffusion}. The upper bound is enough to show that a $\gamma$-LBM, $Z^\gamma$, spends Lebesgue-almost all of its time in $\gamma$-thick points, almost surely. Our trick, therefore, is to construct two Liouville Brownian motion processes simultaneously on the same underlying path $B$; one will use the parameter $\gamma$, the other will use the parameter $\alpha$. Then, if we sample a time uniformly at random and look at where the process $Z^\alpha$ is, it will almost surely be an $\alpha$-thick point. Since the process is constructed using the Brownian motion $B$, we know that $B$ must pass through that particular $\alpha$-thick point at some time $t$, say. But then we know that $F_\gamma(t)$ corresponds to a time that $Z^\gamma$ is in an $\alpha$-thick point.

Using this procedure, we can construct a measure on the (Euclidean) set of times that $Z^\gamma$ spends in $\alpha$-thick points. We can then sample a time at random from this measure, and look at the regularity properties of $F_\gamma$ around that time. This idea is more thoroughly fleshed out in \cref{subsec:holder}.

\section{Setup}\label{section:setup}

We will now collect a few of the definitions and results that we use throughout \cref{sec:proofs}. Throughout, we let $\mc D$ be a simply connected, proper domain in $\mb C$. By conformal invariance of Liouville Brownian motion (including its clock process) we can assume without loss of generality that $\mc D$ is bounded. (See Theorem 1.3 in \cite{berestycki2013diffusion}.)

\subsection{Gaussian Free Field}\label{subs:gff}
We will briefly introduce the Gaussian Free Field here, mostly to clarify our notation. For more detail see, for example, \cite{sheffield2007gaussian} or the introduction of \cite{duplantier2011liouville}. 

Before we can define the GFF we need to define the Dirichlet inner product.  For any two smooth, compactly supported functions $\phi$ and $\psi$ defined on $\mc D$, we define the Dirichlet inner product as
\begin{equation*}
	\dipd{\phi}{\psi} = \frac{1}{2\pi} \int_{\mc D} \grad \phi(z) \cdot \grad \psi(z) dz.
\end{equation*}
We can now define the Gaussian Free Field. 

\begin{definition}
	Let $H^1_0(\mc D)$ be the Sobolev space given by the completion under the Dirichlet inner product of smooth, compactly supported functions defined on $\mc D$. The Gaussian Free Field is a centered Gaussian process on the space $H^1_0(\mc D)$. 
\end{definition}

A consequence of the Hilbert space definition given above is that for any two functions $f, g \in H^1_0(\mc D)$, the random variables $\dipd{h}{f}$ and $\dipd{h}{g}$ are centred Gaussian random variables with covariance
\begin{equation*}
	\cov{\left(\dipd{h}{f},\dipd{h}{g}\right)} = \dipd{f}{g}.
\end{equation*}
This means that we can define a regularisation of the GFF, and we know about its covariance properties.  

\begin{definition}
	The average of the GFF $h$ on a circle of radius $\varepsilon$, centred at a point $z \in \left\{ z' \in \mc D \, : \, \text{dist}( z', \partial \mc D) > \varepsilon \right\}$ is defined as as
	\begin{equation*}
		h_\varepsilon(z) = \dipd{h}{\xi^z_\varepsilon}.
	\end{equation*}
	The function $\xi^z_\varepsilon$ is given by
	\begin{equation}
		\xi^z_\varepsilon(y) = -\log (|z-y|\vee \varepsilon) + \overline{\phi^z_\varepsilon}(y),
		\label{eq:xi}
	\end{equation}
	where $\overline{\phi^z_\varepsilon}$ is harmonic in $\mc D$ and is equal to $\log (|z-y| \vee \varepsilon)$ for $y \in \partial \mc D$. 
\end{definition}

The reason that we think of the above definition as giving the circle average of the GFF is that, as a distribution, we have that $-\Delta \xi^z_\varepsilon = 2\pi\nu^z_\varepsilon$, where $\nu^z_\varepsilon$ is the uniform distribution on the circle centred at $z$ with radius $\varepsilon$. Therefore, integration by parts gives us
\begin{equation*}
	\dipd{h}{\xi^z_\varepsilon} = \ipd{h}{\nu^z_\varepsilon},
\end{equation*}
where $\ipd{\cdot}{\cdot}$ refers to the standard $L^2$ inner product. We will use a continuous modification of the circle average process $\left\{ h_\varepsilon(z) \; ; \; \varepsilon > 0 \right\}_{z \in \mc D}$ throughout. For more detail, see Propositions 3.1 and 3.2 in \cite{duplantier2011liouville}.

The following lemma will be useful in \cref{subsec:scaling}, as it allows us to use properties of the $\log$ function rather than relying on the abstract definition of $\e{h_\varepsilon(x)h_\eta(y)}$. 

\begin{lemma}
	Let $h$ be a zero boundary GFF defined on a simply connected domain $\mc D$. For any subdomain $\tilde{ \mc D}$ which is compactly contained in $\mc D$, there exists a constant $C > 0$ such that, for all $0< \varepsilon, \eta  \leq \text{dist}(\tilde{\mc D}, \partial \mc D)$ with $\eta \leq \varepsilon$,
	\begin{equation*}
		\log \frac{1}{|x-y| + \varepsilon} - C \leq \e{h_\varepsilon(x) h_\eta(y)} \leq \log \frac{1}{|x-y| + \varepsilon} + C
	\end{equation*}
	for all $x, y \in \tilde{ \mc D}$. 
	\label{lem:cov.bound}
\end{lemma}

\begin{proof}
	First note that, by definition, $\mb E [h_\varepsilon(x)h_\eta(y)] = \dipd{\xi^x_\varepsilon}{\xi^y_\eta}$. Integration by parts lets us write that as $\dipd{\xi^x_\varepsilon}{\xi^y_\eta} = \ipd{\xi^x_\varepsilon}{\nu^y_\eta}$.
	Since $x,y \in \tilde{\mc D}$ are uniformly bounded away from $\partial \mc D$, $\mc D$ is a bounded domain, and $\overline{\phi^x_\varepsilon}$ (defined in \cref{eq:xi}) is harmonic in $\tilde{\mc D}$, we know that there exists a constant $\overline{C}$ such that
	\begin{equation}
		-\overline{C} \leq \overline{\phi^x_\varepsilon} (y) \leq \overline{C}
		\label{eq:phi.bound}
	\end{equation}
	for all $x,y \in \tilde {\mc D}$ and $\varepsilon > 0$. So, to complete the proof, it is sufficient to find bounds on $\ipd{-\log(|x-\cdot|\vee\varepsilon)}{\nu^y_\eta}$. To that end, we claim that, for all $x,y \in \mc D$ and $u \in \mc D$ such that $|u - y| \leq \eta$, we have
	\begin{equation}
		\frac{1}{3}\left( |x - y| + \varepsilon \right) \leq |x - u| \vee \varepsilon \leq |x - y| + \varepsilon.
		\label{eq:tri.bound}
	\end{equation}
	The right hand inequality follows directly from the triangle inequality. For the left hand inequality, note that 
	\begin{equation*}
		\frac{1}{3}\left( |x - y| + \varepsilon \right) \leq \frac{1}{3}|x - u| + \frac{2}{3}\varepsilon
	\end{equation*}
	by the triangle inequality. Then, if $|x - u| \leq \varepsilon$, we see that
	\begin{equation*}
		\frac{1}{3}|x - u| + \frac{2}{3}\varepsilon \leq \varepsilon = |x - u| \vee \varepsilon,
	\end{equation*}
	and if $|x - u| > \varepsilon$, we see that
	\begin{equation*}
		\frac{1}{3}|x - u| + \frac{2}{3}\varepsilon \leq |x - u| = |x - u| \vee \varepsilon.
	\end{equation*}

	Now, the inequalities in \cref{eq:tri.bound} imply that, for all $x,y \in \tilde{\mc D}$ and $u \in \partial B(y,\eta)$, there exists some constant $\tilde C$ such that
	\begin{equation*}
	-\log(|x-y| + \varepsilon) - \tilde C \leq	-\log\left( |x - u| \vee \varepsilon \right) \leq -\log(|x - y| + \varepsilon) + \tilde C.
	\end{equation*}
	When we average over $u \in \partial B(y, \eta)$ therefore, we find that
	\begin{equation*}
		-\log(|x-y| + \varepsilon) - \tilde C \leq	\ipd{-\log(|x - \cdot|\vee \varepsilon)}{\nu^z_\eta} \leq -\log(|x - y| + \varepsilon) + \tilde C,
	\end{equation*}
	which, when we combine it with \cref{eq:phi.bound}, completes the proof.
\end{proof}

One of the properties of the Gaussian Free Field which we will use is the domain Markov property. It roughly states that, given a subdomain $\mc U \subset \mc D$, the GFF $h$ on $\mc D$ can be decomposed as the sum of a zero boundary GFF $\tilde h$ on $\mc U$ and the difference, $h^{har} = h - \tilde h$, which is independent of $\tilde h$ and harmonic on $\mc U$. 

\begin{proposition}[Markov property]\label{prop:markov}
	Let $\mc U \subset \mc D$ be a subdomain of the simply connected domain $\mc D$. Let $h$ be a GFF on $\mc D$. Then we can write $h = h^{har} + \tilde h$, where
  \begin{enumerate}
		\item $h^{har}$ and $\tilde h$ are independent,
    \item $\tilde h$ is a zero boundary GFF on $\mc U$ and zero on $\mc D \setminus \mc U$, 
		\item $h^{har}$ is harmonic on $\mc U$ and agrees with $h$ on $\mc D \setminus \mc U$. 
  \end{enumerate}
\end{proposition}

\begin{note}
	We will often refer to $h^{har}$ in the decomposition above as ``the harmonic projection of $h$ onto $\mc U$.''
\end{note}

We now define the set of $\alpha$-thick points of the field $h$. We can think of these as a kind of ``level set'' of the field. We are interested in how much time the Liouville Brownian motion spends in these points, for $\alpha \in [0,2)$ in particular.

\begin{definition}
	The set of $\alpha$-thick points, $\mc T_\alpha$ is
	\begin{equation*}
		\mc T_\alpha = \left\{ z \in \mc D \, : \, \lim_{\varepsilon \to 0}\frac{h_\varepsilon(z)}{-\log \varepsilon} = \alpha \right\}.
	\end{equation*}
\end{definition}

\subsection{Scale Invariant Gaussian Field}\label{subs:log.field}
To help with calculations in \cref{subs:moments}, we introduce a centred Gaussian field $Y$ defined on the whole complex plane, following the presentation of \cite{rhodes2013gaussian}. We use this particular log-correlated field because it has the exact stochastic scale invariance property. A great deal more information about log-correlated Gaussian fields and the measures created from them (those of Gaussian multiplicative chaos) can be found in \cite{kahane1985chaos,robert2010gaussian} for example, and more about the scaling relations which log-normal random measures satisfy can be found in \cite{allez2013lognormal}.

Informally, we define field $Y$ to be a centered Gaussian field on $\mb C$ with covariance function
	\begin{equation*}
		\e{Y(x) Y(y)} = \log_+ \frac{T}{|x-y|} + C.
	\end{equation*}
	for positive constants $T$ and $C$. For simplicity, we will take $T = 1$ and $C = 0$ throughout. We give the precise definition using the white noise decomposition of the field:

\begin{definition}
	Let $(Y_\varepsilon)_{\varepsilon \in (0,1]}$ be the white noise decomposition of the field $Y$, which has correlation structure 
	\begin{equation*}
		\mb E\left[ Y_\varepsilon(x) Y_\varepsilon(y) \right] = 
		\begin{cases}
			0 & \text{if } |x-y| > 1 \\
			\log \frac{1}{|x-y|} & \text{if } \varepsilon \leq |x-y| \leq 1 \\
			\log \frac{1}{\varepsilon} + 2\left( 1 - \frac{|x-y|^{\frac{1}{2}}}{\varepsilon^{\frac{1}{2}}} \right) & \text{if } |x-y| \leq \varepsilon.
		\end{cases}
	\end{equation*}
\end{definition}

The following lemma is useful in the study of properties Gaussian multiplicative chaos locally in $Y$.

\begin{lemma}[Exact stochastic scale invariance]
	For all $\lambda < 1$, the field $Y$ satisfies the following scaling relation:
	\begin{equation*}
		(Y_{\lambda \varepsilon}(\lambda x))_{|x| \leq \frac{1}{2}} \overset{d}{=} (Y_\varepsilon(x))_{|x| \leq \frac{1}{2}} + \Omega_{\lambda},
	\end{equation*}
	where $\Omega_\lambda$ is a centred Gaussian random variable with variance $\log \frac{1}{\lambda}$, independent of the field $Y$. 
	\label{lem:scale.invar}
\end{lemma}

\subsection{Liouville Brownian motion}\label{subs:lbm}

The Liouville Brownian motion is defined as a time change of a Brownian motion, with the path chosen independently from the field $h$. We will start the Brownian motion at the origin (assuming $0\in \mc D$), and run it until some a.s.~finite stopping time $T$. The following definition is non-trivial: for more details about the almost sure existence of the limit and other properties, see \cite{berestycki2013diffusion,garban2013liouville}.

\begin{definition}
Let $B$ be a planar Brownian motion, independent of the field $h$. For $\varepsilon > 0$ and $\gamma \in [0,2)$, define the regularised time change $F_{\gamma,\varepsilon}$ by
	\begin{equation*}
		F_{\gamma,\varepsilon}(t) = \int_0^{t \wedge T} e^{\gamma h_\varepsilon (B_s) - \frac{\gamma^2}{2} \e{h_\varepsilon(B_s)^2} } ds.
	\end{equation*}
	The time change $F_\gamma$ is defined as the limit
	\begin{equation*}
		F_\gamma(t) = \lim_{\varepsilon \to 0} F_{\gamma,\varepsilon}(t).
	\end{equation*}
	\label{def:time.change}
\end{definition}

\begin{definition}
	Using the same Brownian motion $B$ as in \cref{def:time.change}, we define the $\gamma$-Liouville Brownian motion ($\gamma$-LBM for short) $Z^\gamma$ as
	\begin{equation*}
		Z^\gamma_t = B_{F^{-1}_\gamma (t)}.
	\end{equation*}
\end{definition}

\begin{note}
	If we call $T_\alpha = \left\{ t \geq 0 \, : \, B_t \in \mc T_\alpha \right\}$ the set of times that the Brownian motion $B$ spends in $\alpha$-thick points, the set of times that the $\gamma$-LBM $Z^\gamma$ spends in $\alpha$-thick points is the image, under the map $F_\gamma$, of the times that $B$ spends in them, \ie $F_\gamma(T_\alpha) = \left\{ t\geq 0 \, : \, Z^\gamma_t \in T_\alpha \right\}$.
	\label{note:image}
\end{note}

As the Brownian path $B$ of the Liouville Brownian motion $Z^\gamma$ is independent of the Gaussian Free Field $h$, it will be useful to decompose the probability measure $\mb P$ as
\begin{equation*}
	\mb P = \mb P_B \otimes \mb P_h.
\end{equation*}
Decomposing $\mb P$ in this way will let us consider expectations on events which depend only on the field $h$ or the path $B$. 

\subsection{Hausdorff dimension}\label{subs:hd}
We will now recall the definition of the Hausdorff of a set, and collect some useful tools for finding upper and lower bounds for the Hausdorff dimension. Since we will be working in either $\mb R$ or $\mb R^2$, we will not state the definitions in their full generality. For more detail see, for example, Chapter 4 of \cite{morters2010brownian}. 

\begin{definition}
	Let $E \subset \mb R^n$. For $s \geq 0$ and $\delta > 0$ we define
	\begin{equation*}
		\mc H^s_\delta(E) = \inf \left\{ \sum_{i = 1}^\infty |E_i|^s \; : \; E \subset \bigcup_{i = 1}^\infty E_i \enspace \text{and} \enspace |E_i| < \delta \enspace \forall i \geq 1 \right\},
	\end{equation*}
	where $|E_i| = \sup\left\{ |x-y| \; : \; x,y\in E_i \right\}$ is the diameter of the set $E_i$. Then the limit
	\begin{equation*}
		\mc H^s(E) = \lim_{\delta \to 0} \mc H^s_\delta(E)
	\end{equation*}
	is the $s$-Hausdorff measure of $E$. 
\end{definition}

\begin{definition}
	The Hausdorff dimension of a set $E \subset \mb R^n$ is defined as
	\begin{equation*}
		\dim_H(E) = \inf \left\{ s \geq 0 \; : \; \mc H^s(E) = 0 \right\}.
	\end{equation*}
\end{definition}

One tool for finding bounds on the Hausdorff dimension of a set is to use H\"older continuity properties of functions. Indeed, if $f:\mb R^n \rightarrow \mb R^m$ is $\beta$-H\"older continuous, then for any set $E \subset \mb R^n$ we have
\begin{equation}
	\dim_H(f(E)) \leq \frac{1}{\beta}\dim_H(E),
	\label{eq:hausdorff}
\end{equation}
where $f(E) = \left\{ f(x) \; : \; x \in E \right\}$ is the image of $E$ under $f$. The assumption of H\"older continuity is too strong for our purpose. We now define what we call a $\beta$-H\"older-like function, and show that the property is strong enough that the inequality in \cref{eq:hausdorff} still holds. 

\begin{definition}\label{defn:holder.like}
	Let $f:\mb R \rightarrow \mb R$ be a continuous function, and let $E \subset \mb R$. We say that $f$ is $\beta$-H\"older-like on $E$ if there exist constants $C, R > 0$ such that
	\begin{equation*}
		|f(x) - f(x+r)| \leq C r^\beta
	\end{equation*}
	for all $r \in [0,R)$ and $x \in E$.
\end{definition}

\begin{proposition}
	Let $E \subset \mb R$, and suppose that $f:\mb R \rightarrow \mb R$ is increasing and $\beta$-H\"older-like on $E$. Then we have the bound
	\begin{equation*}
		\dim_H(f(E)) \leq \frac{1}{\beta} \dim_H(E).
	\end{equation*}
	\label{prop:holder}
\end{proposition}

\begin{proof}
	Suppose that the radius and multiplicative constant for the H\"older-like property of $f$ are $R$ and $C$ respectively. 
	Let $s > \dim_H(E)$, and let $\varepsilon > 0$. Since $\mc H^s(E) = 0$, we know that there exists some $\delta_0$ such that $\mc H^s_\delta(E) \leq \varepsilon$ for all $\delta\in(0, \delta_0)$. 

	Fix a particular $\delta \in (0, \delta_0 \wedge R)$. Then we can find a cover $\left\{ E_i \right\}$ of $E$ with $|E_i| < \delta$ for all $i$ such that
	\begin{equation*}
		\sum_{i = 1}^\infty |E_i|^{s} < \varepsilon.
	\end{equation*}
	Without loss of generality, we may assume that the intersection $E \cap E_i$ is non-empty. Therefore, for each $E_i$ we can define an interval $I_i = [a_i, b_i]$, where $a_i = \inf \left\{ E_i \cap E \right\}$ and $b_i = \sup\left\{ E_i \cap E \right\}$. Then certainly $|I_i| < \delta$ for all $i$, the sets $\left\{ I_i \right\}$ cover $E$ and $\sum |I_i| < \varepsilon$.

	As $f$ is increasing, we know that
	\begin{equation*}
	|f(I_i)| = |f(a_i) - f(b_i)|.
	\end{equation*}
	Now, $a_i$ is a limit point of $E_i \cap E$, so we can find a sequence $\left\{ x_n \right\} \subset E_i \cap E$ such that $x_n \to a_i$ as $n \rightarrow \infty$. Since each $x_n \in E$, the $\beta$-H\"older-like property of $f$, tells us that
	\begin{equation*}
		|f(I_i)| \leq |f(a_i) - f(x_n)| + |f(x_n) - f(b_i)| \leq |f(a_i) - f(x_n)| + C|x_n - b_i|^\beta.
	\end{equation*}
	So, letting $n \rightarrow \infty$, and recalling that $f$ is continuous (by assumption), we see that
	\begin{equation*}
		|f(I_i)| \leq C|a_i - b_i|^\beta = C |I_i|^\beta.
	\end{equation*}
	Therefore, we can deduce that
	\begin{equation*}
		\sum_{i = 1}^\infty |f(I_i)|^{\frac{s}{\beta}} \leq C \sum_{i = 1}^\infty |I_i|^s < C\varepsilon.
	\end{equation*}

	Since $\left\{ f(I_i) \right\}$ covers $f(E)$, we have shown that
	\begin{equation*}
		\mc H^{\frac{s}{\beta}}_{2\delta}(f(E)) < C\varepsilon.
	\end{equation*}
	We may now let $\delta \downarrow 0$ and then $\varepsilon \downarrow 0$ to see that $\mc H^{\frac{s}{\beta}}(f(E)) = 0$, and hence 
	\begin{equation*}
		\dim_H(f(E)) \leq \frac{s}{\beta}.
	\end{equation*}
	Now letting $s \downarrow \dim_H(E)$ gives the desired result. 
\end{proof}

\section{Proofs of the Main Theorems}\label{sec:proofs}
One of the tools we use in the proof of the lower bound is the exact stochastic scale invariance of the auxiliary field $Y$. In order to do that, we need to ensure that we consider times when the Brownian motion $B$ does not stray too far from the origin. Therefore, we define the stopping time $\tau = \inf \left\{ t \geq 0 \; : \; B_t \not \in B(0,\frac{1}{2}) \right\}$, where $B(0,\frac{1}{2})$ is the ball of radius $\frac{1}{2}$ centred at the origin. For simplicity, we will assume that our domain contains the ball of radius $\frac{1}{2}$, $B(0,\frac{1}{2}) \subset \mc D$.

\subsection{Moments of $F_\gamma$ around a thick point}\label{subs:moments}

We first need to obtain estimates on the moments of the time change $F_\gamma$ around $\alpha$-thick points of the free field. We will use these bounds in \cref{subsec:holder} to derive H\"older-like properties of $F_\gamma$.

Since the law of the GFF conditional on the origin being a thick point is that of an independent zero boundary GFF plus a log singularity, $h(z) \overset{d}{=} \tilde h(z) - \alpha \log |z|$, the effect on the measure is to divide by $|z|^{\alpha\gamma}$. That is why we are thinking the results in this section as results about $F_\gamma$ close to thick points. 

\begin{proposition}[A positive moment is bounded]
	Let $\alpha, \gamma \in [0,2)$. Then, for some $p>0$, there exists a finite constant $C_p$ such that
	\begin{equation*}
		\sup_{\varepsilon\in[0,1)} \e{\left( \int_0^\tau \frac{e^{\gamma h_\varepsilon(B_s) - \frac{\gamma^2}{2}\e{h_\varepsilon(B_s)^2}}}{\left( |B_s| + \varepsilon \right)^{\alpha\gamma}} ds \right)^p} \leq C_p.
	\end{equation*}
	\label{prop:positive}
\end{proposition}

\begin{proof}
	By Kahane's convexity inequality (Lemma 2 in \cite{kahane1985chaos}), taking the measure $\nu(ds) = \frac{ds}{(|B_s| + \varepsilon)^{\alpha\gamma}}$, it is sufficient to prove the proposition for the scale invariant field $Y$.

	Let $\sigma$ be the first time that $B$ leaves the disc of radius $\frac{1}{2\sqrt 2}$. Then, by subadditivity of $x \mapsto x^p$ for $p \in (0,1)$, we know that
	\begin{align}
		\nonumber\mb E \left[ \left( \int_0^\tau \frac{e^{\gamma Y_\varepsilon(B_s) - \frac{\gamma^2}{2}\e{Y_\varepsilon(B_s)^2}}}{\left( |B_s| + \varepsilon \right)^{\alpha\gamma}} ds \right)^p \right]
		& \leq \mb E \left[ \left( \int_0^\sigma \frac{e^{\gamma Y_\varepsilon(B_s) - \frac{\gamma^2}{2}\e{Y_\varepsilon(B_s)^2}}}{\left( |B_s| + \varepsilon \right)^{\alpha\gamma}} ds \right)^p \right] \\
		& + \mb E \left[ \left( \int_\sigma^\tau \frac{e^{\gamma Y_\varepsilon(B_s) - \frac{\gamma^2}{2}\e{Y_\varepsilon(B_s)^2}}}{\left( |B_s| + \varepsilon \right)^{\alpha\gamma}} ds \right)^p \right],
		\label{eq:phone}
	\end{align}
	and it is sufficient to find a uniform upper bound for the right hand side. 

	We will first find a uniform bound for the second term on the right hand side of \cref{eq:phone}. Let $R < \frac{1}{2\sqrt 2}$ be fixed, which we will choose later, and define the time $\tau_R = \inf\left\{ t > \sigma \, : \, |B_t| \leq R \right\}$ that the Brownian motion returns to the ball of radius $R$ after it reaches the circle of radius $\frac{1}{2\sqrt 2}$. 

	On the event $\left\{ \tau_R > \tau \right\}$, we know that $|B_t| > R$ for all $t \in (\sigma, \tau)$. Therefore we find the bound
	\begin{align}
		\nonumber \mb E & \left[ \left( \int_\sigma^\tau \frac{e^{\gamma Y_\varepsilon(B_s) - \frac{\gamma^2}{2}\e{Y_\varepsilon(B_s)^2}}}{\left( |B_s| + \varepsilon \right)^{\alpha\gamma}} ds \right)^p \ind_{\{\tau_R > \tau\}} \right] \leq \\
		&\nonumber \hspace{2em} \leq R^{-\alpha\gamma p} \mb E \left[ \left( \int_\sigma^\tau e^{\gamma Y_\varepsilon(B_s) - \frac{\gamma^2}{2} \e{ Y_\varepsilon(B_s)^2 }} ds \right)^p \ind_{\{\tau_R > \tau\}} \right] \\
		&\hspace{2em} \leq R^{-\alpha\gamma p} \mb E \left[ \left( \int_\sigma^\tau e^{\gamma Y_\varepsilon(B_s) - \frac{\gamma^2}{2} \e{ Y_\varepsilon(B_s)^2 }} ds \right)^p \right]
		\label{eq:ticket}
	\end{align}
	Now, the $L^1$ norm of the regularised change of time process is uniformly bounded in $\varepsilon$ and, since  $p<1$, the expectation of the $p^{th}$ power of it on the event must also be uniformly bounded in $\varepsilon$. We will call the uniform bound $M$.  

	On the event $\left\{ \tau_R < \tau \right\}$ we will split up the interval $\left( \sigma, \tau \right)$ into $(\sigma, \tau_R)$ and $(\tau_R, \tau)$, using the sub-additivity of $x \mapsto x^p$ as before to find that
	\begin{align}
		\nonumber\mb E & \left[ \left( \int_\sigma^\tau \frac{e^{\gamma Y_\varepsilon(B_s) - \frac{\gamma^2}{2}\e{Y_\varepsilon(B_s)^2}}}{\left( |B_s| + \varepsilon \right)^{\alpha\gamma}} ds \right)^p \ind_{ \{\tau_R < \tau \} }\right] \leq \\
		\nonumber&\hspace{4em} \leq \mb E \left[ \left( \int_\sigma^{\tau_R} \frac{e^{\gamma Y_\varepsilon(B_s) - \frac{\gamma^2}{2}\e{Y_\varepsilon(B_s)^2}}}{\left( |B_s| + \varepsilon \right)^{\alpha\gamma}} ds \right)^p \ind_{ \{\tau_R < \tau \} }\right] +  \\
		&\hspace{4em} + \mb E \left[ \left( \int_{\tau_R}^\tau \frac{e^{\gamma Y_\varepsilon(B_s) - \frac{\gamma^2}{2}\e{Y_\varepsilon(B_s)^2}}}{\left( |B_s| + \varepsilon \right)^{\alpha\gamma}} ds \right)^p \ind_{ \{\tau_R < \tau \} }\right],
		\label{eq:coffee}
	\end{align}

	Consider the first term on the right hand side of \cref{eq:coffee}. Similarly to before, we know that $|B_t| > R$ for all $t \in (\sigma, \tau_R)$, and so we can bound the expectation uniformly by
	\begin{align}
		\nonumber \mb E & \left[ \left( \int_\sigma^{\tau_R} \frac{e^{\gamma Y_\varepsilon(B_s) - \frac{\gamma^2}{2}\e{Y_\varepsilon(B_s)^2}}}{\left( |B_s| + \varepsilon \right)^{\alpha\gamma}} ds \right)^p \ind_{\{\tau_R < \tau\}} \right] \leq \\
		&\nonumber \hspace{2em} \leq R^{-\alpha\gamma p} \mb E \left[ \left( \int_\sigma^{\tau_R} e^{\gamma Y_\varepsilon(B_s) - \frac{\gamma^2}{2} \e{ Y_\varepsilon(B_s)^2 }} ds \right)^p \ind_{\{\tau_R < \tau\}} \right] \\
		&\nonumber \hspace{2em} \leq R^{-\alpha\gamma p} \mb E \left[ \left( \int_\sigma^\tau e^{\gamma Y_\varepsilon(B_s) - \frac{\gamma^2}{2} \e{ Y_\varepsilon(B_s)^2 }} ds \right)^p \right]\\
		&\hspace{2em} \leq R^{-\alpha\gamma p}M.
		\label{eq:sofa}
	\end{align}
	To deal with the interval $(\tau_R, \tau)$, let $W$ be another Brownian motion, with $W_0 = B_{\tau_R}$ and which, for $t > 0$, evolves independently of $B$. Let $T = \inf\left\{ t > 0 \; : \; W_t \not \in B(0, \frac{1}{2}) \right\}$. Then, by the strong Markov property of Brownian motion, we see that
	\begin{equation*}
		\left( \int_{\tau_R}^\tau \frac{e^{\gamma Y_\varepsilon(B_s) - \frac{\gamma^2}{2}\e{Y_\varepsilon(B_s)^2}}}{\left( |B_s| + \varepsilon \right)^{\alpha\gamma}} ds \right)^p \ind_{ \{\tau_R < \tau \} }
		\overset{d}{=} \left( \int_0^T \frac{e^{\gamma Y_\varepsilon(W_s) - \frac{\gamma^2}{2}\e{Y_\varepsilon(W_s)^2}}}{\left( |W_s| + \varepsilon \right)^{\alpha\gamma}} ds \right)^p \ind_{ \{\tau_R < \tau \} },
	\end{equation*}
	and that $\int_0^T \frac{e^{\gamma Y_\varepsilon(W_s) - \frac{\gamma^2}{2}\e{Y_\varepsilon(W_s)^2}}}{\left( |W_s| + \varepsilon \right)^{\alpha\gamma}} ds$ is independent of $\ind_{ \left\{ \tau_R < \tau \right\} }$. 
	Therefore, we see that
	\begin{align}
		\nonumber&\e{	\left( \int_{\tau_R}^\tau \frac{e^{\gamma Y_\varepsilon(B_s) - \frac{\gamma^2}{2}\e{Y_\varepsilon(B_s)^2}}}{\left( |B_s| + \varepsilon \right)^{\alpha\gamma}} ds \right)^p \ind_{ \{\tau_R < \tau \} } } = \\
		&\hspace{5em} = \e{ \left( \int_0^T \frac{e^{\gamma Y_\varepsilon(W_s) - \frac{\gamma^2}{2}\e{Y_\varepsilon(W_s)^2}}}{\left( |W_s| + \varepsilon \right)^{\alpha\gamma}} ds \right)^p } \p{\tau_R < \tau}.
		\label{eq:sock}
	\end{align}
	Because we started $W$ closer to the boundary of $B(0, \frac{1}{2})$ we know that, on average, it has a shorter lifespan than $B$, and so
	\begin{equation}
		\e{ \left( \int_0^T \frac{e^{\gamma Y_\varepsilon(W_s) - \frac{\gamma^2}{2}\e{Y_\varepsilon(W_s)^2}}}{\left( |W_s| + \varepsilon \right)^{\alpha\gamma}} ds \right)^p } 
		\leq \e{ \left( \int_0^\tau \frac{e^{\gamma Y_\varepsilon(B_s) - \frac{\gamma^2}{2}\e{Y_\varepsilon(B_s)^2}}}{\left( |B_s| + \varepsilon \right)^{\alpha\gamma}} ds \right)^p }.
		\label{eq:book}
	\end{equation}
	Combining \cref{eq:sock} and \cref{eq:book} gives us the bound
	\begin{align}
		\nonumber &\e{	\left( \int_{\tau_R}^\tau \frac{e^{\gamma Y_\varepsilon(B_s) - \frac{\gamma^2}{2}\e{Y_\varepsilon(B_s)^2}}}{\left( |B_s| + \varepsilon \right)^{\alpha\gamma}} ds \right)^p \ind_{ \{\tau_R < \tau \} } } \leq \\
		&\hspace{4em} \leq\e{ \left( \int_0^\tau \frac{e^{\gamma Y_\varepsilon(B_s) - \frac{\gamma^2}{2}\e{Y_\varepsilon(B_s)^2}}}{\left( |B_s| + \varepsilon \right)^{\alpha\gamma}} ds \right)^p } \p{\tau_R < \tau}.
		\label{eq:flapjack}
	\end{align}

	Now, consider the first term on the right hand side of \cref{eq:phone}. Scaling time by a factor of $\frac{1}{2}$ and space by a factor of $\frac{1}{\sqrt 2}$ gives
	\begin{align*}
	\int_0^\sigma \frac{e^{\gamma Y_\varepsilon(B_s) - \frac{\gamma^2}{2}\e{Y_\varepsilon(B_s)^2}}}{\left( |B_s| + \varepsilon \right)^{\alpha\gamma}} ds
	&= 2^{-1} \int_0^{2\sigma} \frac{e^{\gamma Y_\varepsilon(B_{u/2}) - \frac{\gamma^2}{2}\e{Y_\varepsilon(B_{u/2})^2}}}{\left( |B_{u/2}| + \varepsilon \right)^{\alpha\gamma}} du\\
	&\overset{d}{=} 2^{-(1 - \frac{\alpha\gamma}{2})}\int_0^{\tilde{\tau}} \frac{e^{\gamma Y_\varepsilon(\frac{1}{\sqrt 2} \tilde B_{u}) - \frac{\gamma^2}{2}\e{Y_\varepsilon(\frac{1}{\sqrt 2} \tilde B_{u})^2}}}{\left( |\tilde B_{u}| + \varepsilon\sqrt 2 \right)^{\alpha\gamma}} du\\
	&\overset{d}{=} 2^{-(1 - \frac{\alpha\gamma}{2} + \frac{\gamma^2}{4})} e^{\gamma\Omega_{\sqrt 2}} \int_0^{\tilde{\tau}} \frac{e^{\gamma Y_{\varepsilon\sqrt 2}(\tilde B_{u}) - \frac{\gamma^2}{2}\e{Y_{\varepsilon\sqrt 2}(\tilde B_{u})^2}}}{\left( |\tilde B_{u}| + \varepsilon\sqrt 2 \right)^{\alpha\gamma}} du,
	\end{align*}
	the last line coming from \cref{lem:scale.invar}, where $\tilde B$ is an independent Brownian motion, $\tilde \tau$ is the time that $\tilde B$ leaves the disc of radius $\frac{1}{2}$, and $\Omega_{\sqrt 2}$ is a centred Gaussian random variable with variance $\log \sqrt 2$. Therefore, when we take the $p$th moment, we find
	\begin{align}
		\nonumber \mb E & \left[ \left( \int_0^\sigma \frac{e^{\gamma Y_\varepsilon(B_s) - \frac{\gamma^2}{2}\e{Y_\varepsilon(B_s)^2}}}{\left( |B_s| + \varepsilon \right)^{\alpha\gamma}} ds \right)^p \right]= \\
		& \hspace{4em}= 2^{\frac{\gamma^2}{4} p^2 - (1 - \frac{\alpha\gamma}{2} + \frac{\gamma^2}{4})p} \mb E \left[ \left( \int_0^{\tilde{\tau}} \frac{e^{\gamma Y_{\varepsilon\sqrt 2}(\tilde B_{u}) - \frac{\gamma^2}{2}\e{Y_{\varepsilon\sqrt 2}(\tilde B_{u})^2}}}{\left( |\tilde B_{u}| + \varepsilon\sqrt 2 \right)^{\alpha\gamma}} du \right)^p \right].
		\label{eq:brownie}
	\end{align}
	Let us define a sequence of scales by setting $\varepsilon_n = 2^{-\frac{n}{2}}$, for $n \in \mb N$, and call the expectation
	\begin{equation*}
		E_n = \mb E \left[ \left( \int_0^\tau \frac{e^{\gamma Y_{\varepsilon_n}(B_s) - \frac{\gamma^2}{2}\e{Y_{\varepsilon_n}(B_s)^2}}}{\left( |B_s| + \varepsilon_n \right)^{\alpha\gamma}} ds \right)^p \right].
	\end{equation*}
	Using this notation, we substitute the scaling relation in \cref{eq:brownie}, the uniform bounds in \cref{eq:ticket} and \cref{eq:sofa}, and the inequality in \cref{eq:flapjack} into \cref{eq:phone}, to see that, for $n \geq 1$, 
	\begin{equation}
		E_n \leq 2^{\frac{\gamma^2}{4}p^2 - (1 - \frac{\alpha\gamma}{2} + \frac{\gamma^2}{4})p}E_{n-1} + 2R^{-\alpha\gamma p}M + E_n \mb P[\tau_R < \tau].
		\label{eq:sunglasses}
	\end{equation}
	Upon re-arrangement, the inequality in \cref{eq:sunglasses} becomes
	\begin{equation*}
		E_n \leq \left( \frac{2^{\frac{\gamma^2}{4}p^2 - (1 - \frac{\alpha\gamma}{2} + \frac{\gamma^2}{4})p}}{1 - \mb P [\tau_R < \tau]} \right) E_{n - 1} + \frac{2R^{-\alpha\gamma p}M}{1 - \mb P [\tau_R < \tau]}.
	\end{equation*}
	By choosing $p>0$ first and then $R>0$ fixed and small enough, we can ensure that the factor multiplying $E_{n-1}$, $\frac{2^{\frac{\gamma^2}{4}p^2 - (1 - \frac{\alpha\gamma}{2} + \frac{\gamma^2}{4})p}}{1 - \mb P [\tau_R < \tau]}$, is less than $1$. When we have done this, what we have shown is that, for some $\rho \in (0,1)$ and some constant $\tilde M$, we have $E_n < \rho E_{n-1} + \tilde M$, which implies that the sequence $\left\{ E_n \right\}_{n \in \mb N}$ is bounded.
\end{proof}

\begin{proposition}[A negative moment is bounded]
	Let $\alpha, \gamma \in [0,2)$. Then there exists a finite constant $C_{-1}$ such that
	\begin{equation*}
		\sup_{\varepsilon\in[0,1)} \e{\left( \int_0^{1 \wedge \tau} \frac{e^{\gamma h_\varepsilon(B_s) - \frac{\gamma^2}{2}\e{h_\varepsilon(B_s)^2}}}{\left( |B_s| + \varepsilon \right)^{\alpha\gamma}} \right)^{-1}} \leq C_{-1}.
	\end{equation*}
\end{proposition}

\begin{proof}
	Again, using Kahane's convexity inequality (Lemma 2 of \cite{kahane1985chaos}), it is sufficient to prove this result for the scale invariant field $Y$. We know that $|B_s + \varepsilon| \leq \frac{3}{2}$ for all $t \in (0, \tau)$ and for all $\varepsilon \in [0,1)$, and so we find the bound
	\begin{align}
		\nonumber \sup_{\varepsilon\in[0,1)} & \e{\left( \int_0^{1 \wedge \tau} \frac{e^{\gamma Y_\varepsilon(B_s) - \frac{\gamma^2}{2}\e{Y(B_s)^2}}}{\left( |B_s| + \varepsilon \right)^{\alpha\gamma}} \right)^{-1}} \leq \\
			& \hspace{4em} \leq \frac{3}{2}\sup_{\varepsilon\in[0,1)} \e{\left( \int_0^{1 \wedge \tau} e^{\gamma Y_\varepsilon(B_s) - \frac{\gamma^2}{2}\e{Y(B_s)^2}} \right)^{-1}}. 
			\label{eq:frozen}
	\end{align}
	By Lemmas 2.13 and 2.14 of \cite{garban2013liouville}, the right hand side of \cref{eq:frozen} is finite, and so we are done.
\end{proof}

The following corollaries will be useful in \cref{subsec:holder}. 
\begin{corollary}
  For any power $q$, we have a polynomial bound on the probability
	\begin{equation*}d
		\p{\int_0^{1 \wedge \tau} \frac{e^{\gamma h_\varepsilon(B_s) - \frac{\gamma^2}{2}\e{h_\varepsilon(B_s)^2}}}{(|B_s| + \varepsilon)^{\alpha\gamma}}ds \leq r^q} \leq C_{-1} r^q
	\end{equation*}
	for any $r >0$.
	\label{cor:lower.tail}
\end{corollary}

\begin{corollary}
	There exists some $p>0$ such that, for any $q$, we have a polynomial bound on the probability
	\begin{equation*}
		\p{\int_0^\tau \frac{e^{\gamma h_\varepsilon(B_s) - \frac{\gamma^2}{2}\e{h_\varepsilon(B_s)^2}}}{(|B_s| + \varepsilon)^{\alpha\gamma}}ds \geq r^{-q}} \leq C_{p} r^{pq}
	\end{equation*}
	for any $r > 0$.
	\label{cor:upper.tail}
\end{corollary}

The proof of both of these are simple applications of Markov's inequality.

\subsection{Scaling of a Gaussian Free Field}\label{subsec:scaling}

As well as the polynomial behaviour of the tails of $F_\gamma$ around thick points that we saw in \cref{cor:lower.tail,cor:upper.tail}, we also need a bound on the tail behaviour on the supremum (or infimum) of the harmonic projection of a GFF on a disc of radius $\sqrt r$. 

We were able to use the scale invariant field $Y$ throughout \cref{subs:moments} because the moments we were trying to bound were convex (or concave) functions of a multiplicative chaos measure, and so Kahane's convexity theorem let us change between fields. In \cref{subsec:holder}, however, we need to consider moments of certain integrals of the field, weighted by indicator functions depending on the field itself. Kahane's convexity theorem therefore no longer applies, and we have to work directly with the GFF. So, we need an analogue of the scaling property that the field $Y$ has. 

\begin{lemma}
	Let $\mc D \subset \mb C$ be a bounded proper domain, and let $\tilde{\mc D} \subset \mc D$ be a compactly contained subdomain of $\mc D$. Let $h$ be a zero-boundary condition GFF on $\mc D$. Now, using the Markov property (\cref{prop:markov}), let us write
	\begin{equation*}
		h = h^{har} + \tilde h,
	\end{equation*}
	where $h^{har}$ is the harmonic projection of $h$ onto the disc of radius $2\sqrt r$ and centred at $x \in \tilde{\mc D}$, and $\tilde h$ is an independent, zero-boundary condition GFF on the disc of radius $2\sqrt r$ and centred at $x \in \tilde{\mc D}$. Let
	\begin{equation*}
		\Omega_{\sqrt r}^x = \sup_{z \in B(x, \sqrt r)} h^{har}(z)
	\end{equation*}
	be the supremum of $h^{har}$ on $B(x, \sqrt r)$. Then there exist constants $C, p >0 $ such that, for all $r>0$ small enough,
	\begin{equation*}
		\sup_{x \in \tilde{\mc D}} \mb P \left[ \Omega_{\sqrt r}^x > -\log r \right] \leq C r^p. 
	\end{equation*}
	\label{lem:gff.scaling}
\end{lemma}

\begin{proof}
	The proof is essentially the same as the proof of Kolmogorov's continuity criterion. Instead of taking limits to obtain almost sure results, however, we obtain quantitative estimates which hold with high probability. First, let us fix $x \in \tilde{\mc D}$. 
	From the proof of Proposition 3.1 in \cite{duplantier2011liouville}, we know that there exists some constant $K_2$ such that, for $z,w \in \mc D$ and $\varepsilon > 0$,
	\begin{equation}
		\mb E \left[ |h_{\varepsilon\sqrt r}(z) - h_{\varepsilon\sqrt r}(w)|^2 \right] \leq K_2 |z - w|.
		\label{eq:var.bound}
	\end{equation}
	Now, recall that $h^{har}$ is harmonic on $B(x, 2\sqrt r)$ and so, using the mean value property of harmonic functions, we see that the circle average regularisation of $h^{har}$ is equal to $h^{har}$ on sets inside $B(x, 2\sqrt r)$, i.e.~$h^{har}_{\varepsilon\sqrt r}(z) = h^{har}(z)$ for $\varepsilon$ small enough and $z \in B(x, \sqrt r)$. Therefore, we know that
	\begin{equation}
		h_{\varepsilon \sqrt r}(z) = h^{har}(z) + \tilde h_{\varepsilon \sqrt r}(z).
		\label{eq:har}
	\end{equation}
	Equation \cref{eq:har} and the independence of $h^{har}$ and $\tilde h$, when combined with \cref{eq:var.bound} shows us that
	\begin{equation}
		\mb E \left[ |h^{har}(z) - h^{har}(w)|^2 \right] \leq \mb E \left[ |h_{\varepsilon\sqrt r}(z) - h_{\varepsilon\sqrt r}(w)|^2 \right]  \leq K_2 |z-w|
		\label{eq:har.var.bound}
	\end{equation}
	for $z,w \in B(x, \sqrt r)$. Because the field $h$ does not depend on our choice of $x$, we can see that the middle term in \cref{eq:har.var.bound} is independent of $x$, and therefore so is the constant $K_2$. 
	
	Because $h^{har}(z) - h^{har}(w)$ is a Gaussian random variable, \cref{eq:har.var.bound} implies that for any $\eta > 0$, there exists a constant $K_\eta$ such that 
	\begin{equation}
		\mb E \left[ |h^{har}(z) - h^{har}(w)|^\eta \right] \leq K_\eta |z-w|^{\eta/2}
		\label{eq:har.bound}
	\end{equation}
	for $z,w \in B(x,\sqrt r)$. Again, $K_\eta$ is independent of $x$ by the independence of $K_2$ from $x$. 

	Because the supremum of a harmonic function on a domain is attained at the boundary of that domain, we need only consider $z \in \partial B(0, \sqrt r)$. Therefore, let us set
	\begin{equation*}
		X_t = h^{har}\left(x + \sqrt r e^{2\pi i t}\right),
	\end{equation*}
	for $t \in [0,1]$. From the inequality in \cref{eq:har.bound}, we deduce that for $s,t\in[0,1]$
	\begin{align*}
		\mb E\left[ |X_s - X_t|^\eta\right] & \leq K_\eta| \sqrt r\left( e^{2\pi i s} - e^{2\pi i t} \right) | ^{\eta/2} \\
		& \leq \tilde K_\eta r^{\eta/2} |s-t|^{\eta/2}.
	\end{align*}
	The bound on the covariance structure of $X$ clearly does not depend on the choice of $x$. 

	Now that we have a bound on the $\eta$-moment of the increments, we can use Markov's inequality to say things about the probability that the process $X$ is irregular. So, let $D_n = \left\{ k2^{-n} \; : \; k = 0, 1, \dots 2^n \right\}$ be the set of dyadic points in the unit interval at level $n$. For some power $p$, to be chosen later, we have
	\begin{align*}
		\mb P\left[ |X_{k2^{-n}} - X_{(k+1)2^{-n}}| > 2^{-np} r^{1 / 4} \right] &\leq 2^{np\eta} r^{-\eta /4}  \mb E\left[ |X_{k2^{-n}} - X_{(k+1)2^{-n}}|^\eta \right] \\
		&\leq \tilde K_\eta r^{\eta /4} 2^{np\eta} |k2^{-n} - (k+1)2^{-n}|^{\eta/2}\\
		&\leq \tilde K_\eta r^{\eta /4} 2^{-n(\frac{1}{2}\eta - p\eta)},
	\end{align*}
	for $k = 0, 1, \dots, 2^n -1$. Therefore, a simple union bound shows that
	\begin{equation*}
		\mb P \left[ \sup_k |X_{k2^{-n}} - X_{(k+1)2^{-n}}| > 2^{-np}r^{1/4}  \right] \leq \tilde K_\eta r^{\eta/4} 2^{-n(\frac{1}{2}\eta - p\eta - 1)}.
	\end{equation*}
	If we choose $0 < p < \frac{1}{2}$ and $\eta$ sufficiently large, we find that $q := \frac{1}{2}\eta - p\eta - 1 >0$. Because we have ensured that $q >0$, we can again use the union bound to find that
	\begin{align*}
		\mb P \left[ \sup_{n \geq 0} \sup_k |X_{k2^{-n}} - X_{(k+1)2^{-n}}| > 2^{-np}r^{1/4}  \right] &\leq \tilde K_\eta r^{\eta/4} \sum_{n \geq 0} 2^{-nq}\\
		& = \tilde K_\eta r^{\eta/4} \left( \frac{1}{1-2^{-q}} \right)\\
		& = \overline K_\eta r^{\eta/4}.
	\end{align*}
	So, we see that the event 
	\begin{equation*}
		A := \left\{ X_t \text{ is $p$-H\"older continuous with constant $r^{1/4}$} \right\}
	\end{equation*}
	occurs with probability greater than $1 - \overline K_\eta r^{\eta/2}$. On that event we can see that 
	\begin{equation}
		|\sup_{t}X_t - \inf_t X_t| \leq r^{1/4}.
		\label{eq:sup.inf}
	\end{equation}

	Using \cref{eq:sup.inf}, we can find a bound for $\Omega_{\sqrt r}^x$ in terms of objects we have good control over. Specifically, we have
	\begin{equation*}
		\Omega_{\sqrt r}^x \leq |\sup_{t}X_t - \inf_t X_t| + |\overline X|, 
	\end{equation*}
	where $\overline X$ is the mean value of the process $X_t$. Let us consider that second term. Since $X$ is really just $h^{har}$ on a circle, and $h^{har}$ is harmonic, we can use the mean value theorem to see that $\overline X = h^{har}(x)$. But again, we can apply the mean value theorem to see that $h^{har}(x)$ is really just the average of $h$ on $\partial B(x, 2\sqrt r)$, i.e.~$h^{har}(x) = h_{2\sqrt r}(x)$. So, we have the inequality
	\begin{equation*}
		\mb P \left[ \Omega_{\sqrt r} \geq - \log r \right] \leq \mb P \left[ |\sup_{t}X_t - \inf_t X_t| \geq -\frac{1}{2}\log r \right] + \mb P \left[ h_{2\sqrt r}(x) \geq - \frac{1}{2}\log r \right].
	\end{equation*}
	Because $h_{2\sqrt r}(x) \sim N(0, -\log 2\sqrt r + \log C(x, \mc D))$, we know that the second term on the right hand side decays polynomially in $r$ as $r \to 0$. Furthermore, since the conformal radius, $C(x, \mc D)$, is bounded for $x \in \tilde{\mc D}$, the coefficients we choose in the polynomial bound can be chosen to hold uniformly for all $x \in \tilde{\mc D}$. 
	
	So now let us consider the first term. 
	\begin{align*}
		\mb P \Bigg[ |\sup_{t}X_t - \inf_t X_t| \geq -\frac{1}{2}\log r \Bigg] &=
		\mb P \left[ \left\{ |\sup_{t}X_t - \inf_t X_t| \geq -\frac{1}{2}\log r \right\} \cap A \right] \\
		&+ P \left[ \left\{ |\sup_{t}X_t - \inf_t X_t| \geq -\frac{1}{2}\log r \right\} \cap A^c \right]\\
		&\leq \mb P \left[ r^{1/4} \geq -\frac{1}{2}\log r \right] + \mb P\left[ A^c \right]\\
		&\leq 0 + \overline K_\eta r^{\eta/2},
	\end{align*}
	for $r$ small enough. Since $\overline K_\eta$ does not depend on $x \in \tilde{\mc D}$, we have the desired result.  
\end{proof}

\subsection{H\"older-like properties of $F_\gamma$}\label{subsec:holder}

We will now show the required regularity properties of the time change function $F_\gamma$. It will be convenient to introduce the following measures.

\begin{definition}
	Let $\mu_\gamma$ be the measure on the interval $[0,\tau]$ defined by $\mu_\gamma = \mc L \circ F_\gamma^{-1}$, where $\mc L$ is Lebesgue measure on the interval $[0, F_\gamma(\tau)]$. In other words, for $s,t \in [0, \tau]$ with $s \leq t$, we set
	\begin{equation*}
		\mu_\gamma([s,t]) = F_\gamma(t) - F_\gamma(s).
	\end{equation*}
	Define the measure $\mu_\alpha$ in a similar way, for $\alpha \in [0,2)$.
\end{definition}

\begin{remark}
	To get some intuition behind the next few results, let us think of the measures $\mu_\alpha$ and $\mu_\gamma$ as probability measures for a moment. Then, if we sample a time $t \in [0,\tau]$ according to $\mu_\alpha$, it will almost surely be such that the Brownian motion $B$ is in an $\alpha$-thick point, \ie $t \in T_\alpha$ almost surely, because the $\alpha$-LMB, $Z^\alpha,$ spends Lebesgue-almost all of its time in $\alpha$-thick points. Similar statements hold if we sample a time from $\mu_\gamma$. 
\end{remark}

\begin{proposition}
	For all $\alpha \in [0,2)$, and $\gamma \in [0,2)$, fix $\delta > 0$, and let $\beta = 1 - \frac{\alpha\gamma}{2} + \frac{\gamma^2}{4}$. Define the set of times
	\begin{equation*}
		L^N_\gamma = \left\{ t\in [0,\tau] \, : \, \mu_\gamma([(t, (t+r) \wedge \tau]) \geq r^{\beta+\delta} \quad \forall r \in [0,2^{-N})  \right\}. 
	\end{equation*}
	Then for all $\Delta >0$, which may be random and may depend on $\mu_\alpha([0,\tau])$, there exists some random but almost surely finite $N \in \mb N$ such that
	\begin{equation*}
		\mu_\alpha\left( L^N_\gamma \right) \geq \mu_\alpha\left( [0,\tau] \right) - \Delta.
	\end{equation*}
	\label{prop:lower}
\end{proposition}
This proposition is essentially saying that if $t$ is an $\alpha$-thick point, then the $\mu_\gamma$ mass of an interval of length $r$, starting at $t$, decays more slowly than $r^{\beta+\delta}$. It is \emph{almost} like saying that around $\alpha$-thick points, the map function $F_\gamma^{-1}$ is $\frac{1}{\beta+\delta}$-H\"older continuous.

The proof will rely on the following lemma:

\begin{lemma}
	As before, fix $\delta>0$ and let $\beta = 1 - \frac{\alpha\gamma}{2} + \frac{\gamma^2}{4}$, and let $E>0$ be some positive constant. Then there exist two constants $D>0$ and $q>0$ such that
	\begin{equation*}
		\e{\mu_\alpha\left( \left\{ t\in[0,\tau] \, : \, \mu_\gamma([t, (t+r) \wedge \tau]) < E r^{\beta+\delta} \right\} \right) } \leq D r^{q}.
	\end{equation*}
	for all $r \in (0,1)$. 
	\label{lem:size.bound}
\end{lemma}

\begin{proof}
	To ease notation, we will prove the case $E = 1$. The reader will be able to see that the same argument works for any positive $E$, with possibly different constants $D$ and $q$. 
	
	Let $r \geq 0$ and fix $\varepsilon, \varepsilon' >0$ so that $\varepsilon' < \varepsilon\sqrt r$. Then by Girsanov's change of measure theorem, we get
	\begin{align}
		\nonumber &\mb E_B \mb E_h \Big [ \ind_{ \big\{ \int_{t}^{(t+r) \wedge \tau} e^{\gamma h_{\varepsilon\sqrt r}(B_s) - \frac{\gamma^2}{2}  \mb E [ h_{\varepsilon\sqrt r}(B_s)^2 ]}ds < r^{\beta+\delta} \big\}}  e^{\alpha h_{\varepsilon '}(B_t) - \frac{\alpha^2}{2} \e{h_{\varepsilon'}(B_t)^2}} \ind_{ \left\{ \tau > t \right\} } \Big ] = \\[1em]
		\nonumber &= \mb E_B \left[ \ind_{ \left\{ \tau > t \right\} }  \mb P_h\left[ \int_{t}^{(t+r)\wedge \tau} e^{\gamma\left(h_{\varepsilon\sqrt r}(B_s) + \alpha \mb E_h \left[ h_{\varepsilon\sqrt r}(B_s)h_{\varepsilon'}(B_t)\right]\right) - \frac{\gamma^2}{2}\e{h_{\varepsilon\sqrt r}(B_s)^2}} ds < r^{\beta + \delta} \right] \right] \\
		&= \mb E_B \left[ \ind_{ \left\{ \tau > t \right\} } \mb E_B \left[ \mb P_h\left[ \int_{t}^{(t+r)\wedge \tau} e^{\gamma\left(h_{\varepsilon\sqrt r}(B_s) + \alpha \mb E_h \left[h_{\varepsilon\sqrt r}(B_s)h_{\varepsilon'}(B_t)\right]\right) - \frac{\gamma^2}{2}\e{h_{\varepsilon\sqrt r}(B_s)^2}} ds < r^{\beta + \delta} \right] \bigg| \mc F_t \right] \right]
		\label{eq:girsanov}
	\end{align}
	where $\mc F_t = \sigma(B_s; \; s \leq t)$ is the natural filtration for $B$. Now, using \cref{lem:cov.bound}, we know that almost surely on the event $s,t < \tau$, there is a constant $C$ such that 
	\begin{equation*}
		\mb E_h \left[ h_{\varepsilon\sqrt r}(B_s) h_{\varepsilon'}(B_t) \right] \geq \log \frac{1}{|B_s - B_t| + \varepsilon\sqrt r} - C,
	\end{equation*}
	and so we can bound the integral in \cref{eq:girsanov} from below by 
	\begin{align}
		\nonumber\int_{t}^{(t+r)\wedge \tau} & e^{\gamma\left(h_{\varepsilon\sqrt r}(B_s) + \alpha \mb E_h \left[h_{\varepsilon\sqrt r}(B_s)h_{\varepsilon'}(B_t)\right]\right) - \frac{\gamma^2}{2}\e{h_{\varepsilon\sqrt r}(B_s)^2}} ds \geq \\
		&\hspace{6em} \geq e^{-\alpha\gamma C}\int_{t}^{(t+r)\wedge \tau} \frac{ e^{\gamma h_{\varepsilon\sqrt r}(B_s) - \frac{\gamma^2}{2}\e{h_{\varepsilon\sqrt r}(B_s)^2}}}{(|B_s - B_t| + \varepsilon \sqrt r)^{\alpha\gamma}} ds
		\label{eq:integral.bound}
	\end{align}
	Now, by changing variables and using the scaling properties of Brownian motion we see that the right hand side of \cref{eq:integral.bound} becomes
	\begin{align*}
		\nonumber \int_{t}^{(t+r)\wedge \tau} \frac{ e^{\gamma h_{\varepsilon\sqrt r}(B_s) - \frac{\gamma^2}{2}\e{h_{\varepsilon\sqrt r}(B_s)^2}}}{(|B_s - B_t| + \varepsilon \sqrt r)^{\alpha\gamma}} ds 
		&=\int_{0}^{r\wedge (\tau-t)} \frac{ e^{\gamma h_{\varepsilon\sqrt r}(B_{t+s}) - \frac{\gamma^2}{2}\e{h_{\varepsilon\sqrt r}(B_{t+s})^2}}}{(|B_{t+s} - B_t| + \varepsilon \sqrt r)^{\alpha\gamma}} ds \\
		&= r\int_{0}^{1\wedge ( (\tau - t)/r)} \frac{ e^{\gamma h_{\varepsilon\sqrt r}(B_{t + ru}) - \frac{\gamma^2}{2}\e{h_{\varepsilon\sqrt r}(B_{t + ru})^2}}}{(|B_{t + ru} - B_t| + \varepsilon \sqrt r)^{\alpha\gamma}} du \\
		\nonumber&\overset{d}{=} r\int_{0}^{1\wedge \tau'} \frac{ e^{\gamma h_{\varepsilon\sqrt r}(\sqrt r \tilde B_{u} + B_t) - \frac{\gamma^2}{2}\e{h_{\varepsilon\sqrt r}(\sqrt r \tilde B_{u} + B_t)^2}}}{(|\sqrt r \tilde B_{u}| + \varepsilon \sqrt r)^{\alpha\gamma}} du\\
		&= r^{1 - \frac{\alpha\gamma}{2}} \int_{0}^{1 \wedge \tau'} \frac{ e^{\gamma h_{\varepsilon\sqrt r}(\sqrt r \tilde B_{u} + B_t) - \frac{\gamma^2}{2}\e{h_{\varepsilon\sqrt r}(\sqrt r \tilde B_{u} + B_t)^2}}}{(| \tilde B_{u}| + \varepsilon)^{\alpha\gamma}} du,
	\end{align*}
	where $\tilde B$ is an independent Brownian motion started at the origin, and 
	\begin{equation*}
		\tau' = \inf \left\{ u>0 : |\sqrt r \tilde B_u + B_t| = \frac{1}{2}  \right\} 
	\end{equation*}
	is the first time $\sqrt r \tilde B_u + B_t$ exits the disc of radius $1$. The equality in distribution holds $\mb P_h$-almost surely. In order to use the scaling property of the GFF $h$, (\cref{lem:gff.scaling}), we also need to make sure that $\tilde B$ stays bounded. So, let
	\begin{equation*}
		\tilde \tau = \tau' \wedge \inf\left\{ u > 0 : |\tilde B_u| = \frac{1}{2} \right\}.
	\end{equation*}
	Then we certainly know that 
	\begin{align}
	\nonumber r^{1 - \frac{\alpha\gamma}{2}} \int_{0}^{1 \wedge \tau'} & \frac{ e^{\gamma h_{\varepsilon\sqrt r}(\sqrt r \tilde B_{u} + B_t) - \frac{\gamma^2}{2}\e{h_{\varepsilon\sqrt r}(\sqrt r \tilde B_{u} + B_t)^2}}}{(| \tilde B_{u}| + \varepsilon)^{\alpha\gamma}} du	\geq \\
	&\geq r^{1 - \frac{\alpha\gamma}{2}} \int_{0}^{1 \wedge \tilde\tau} \frac{ e^{\gamma h_{\varepsilon\sqrt r}(\sqrt r \tilde B_{u} + B_t) - \frac{\gamma^2}{2}\e{h_{\varepsilon\sqrt r}(\sqrt r \tilde B_{u} + B_t)^2}}}{(| \tilde B_{u}| + \varepsilon)^{\alpha\gamma}} du.
		\label{eq:bm.scaling}
	\end{align}
	
	Now let us use the fact that we are conditioning on $\mc F_t$ and the Markov property of $h$ to write $h = h^{har} + \tilde h$, where $h^{har}$ is the harmonic projection of $h$ onto the disc of radius $2\sqrt r$, centred at $B_t$, and $\tilde h$ has the law of a zero-boundary GFF on the disc of radius $2\sqrt r$, centred at $B_t$. If we write $\Omega_{\sqrt r} = \inf_{z \in B(B_t, \sqrt r)} h^{har}(z)$, we know that $h \geq \Omega_{\sqrt r} + \tilde h$ inside the $B(B_t, \sqrt r)$, and so we can continue from \cref{eq:bm.scaling} to see that 
	\begin{align}
		\nonumber r^{1 - \frac{\alpha\gamma}{2}} & \int_{0}^{1 \wedge \tilde \tau} \frac{ e^{\gamma h_{\varepsilon\sqrt r}(\sqrt r \tilde B_{u} + B_t) - \frac{\gamma^2}{2}\e{h_{\varepsilon\sqrt r}(\sqrt r \tilde B_{u} + B_t )^2}}}{(| \tilde B_{u}| + \varepsilon)^{\alpha\gamma}} du \geq \\
		\nonumber &\geq r^{1 - \frac{\alpha\gamma}{2}} \int_{0}^{1 \wedge \tilde \tau} \frac{ e^{\gamma h_{\varepsilon\sqrt r}(\sqrt r \tilde B_{u} + B_t) - \frac{\gamma^2}{2}\left( - \log(\varepsilon\sqrt r) + C \right)  }}{(| \tilde B_{u}| + \varepsilon)^{\alpha\gamma}} du \\
		\nonumber &= r^{1 - \frac{\alpha\gamma}{2} + \frac{\gamma^2}{4}} e^{-\frac{\gamma^2}{2}C} \int_{0}^{1 \wedge \tilde \tau} \frac{ e^{\gamma h_{\varepsilon\sqrt r}(\sqrt r \tilde B_{u} + B_t) + \frac{\gamma^2}{2}\log\varepsilon  }}{(| \tilde B_{u}| + \varepsilon)^{\alpha\gamma}} du \\
		\nonumber &\geq r^{1 - \frac{\alpha\gamma}{2} + \frac{\gamma^2}{4}} e^{-\frac{\gamma^2}{2}C} e^{\gamma \Omega_{\sqrt r}} \int_{0}^{1 \wedge \tilde \tau} \frac{ e^{\gamma \tilde h_{\varepsilon\sqrt r}(\sqrt r \tilde B_{u} + B_t) + \frac{\gamma^2}{2}\log\varepsilon  }}{(| \tilde B_{u}| + \varepsilon)^{\alpha\gamma}} du \\
		&\geq r^{1 - \frac{\alpha\gamma}{2} + \frac{\gamma^2}{4}} e^{-\gamma^2 C} e^{\gamma \Omega_{\sqrt r}} \int_{0}^{1 \wedge \tilde \tau} \frac{ e^{\gamma h'_{\varepsilon}(\tilde B_{u}) - \frac{\gamma^2}{2} \e{h'_\varepsilon(\tilde B_u)^2}  }}{(| \tilde B_{u}| + \varepsilon)^{\alpha\gamma}} du
		\label{eq:y.scaling}
	\end{align}
	where $h'$ is a zero boundary GFF on the disc of radius $\frac{1}{2}$. Substituting the last expression of \cref{eq:y.scaling} back into \cref{eq:girsanov} (and noticing that the exponent of $r$ is in fact $\beta$) lets us see that
	\begin{align}
		\nonumber \mb P & \left[ \int_{t}^{(t+r) \wedge \tau} e^{\gamma\left(h_{\varepsilon\sqrt r}(B_s) + \alpha \mb E_h \left[h_{\varepsilon\sqrt r}(B_s)h_{\varepsilon'}(B_t)\right]\right) - \frac{\gamma^2}{2}\e{h_{\varepsilon\sqrt r}(B_s)^2}} ds < r^{\beta + \delta} \bigg| \mc F_t \right] \leq \\
		\nonumber &\leq \mb P \left[ e^{-(\alpha\gamma + \gamma^2) C} r^\beta e^{\gamma\Omega_{\sqrt r}} \int_{0}^{1 \wedge \tilde \tau} \frac{ e^{\gamma h'_{\varepsilon}( \tilde B_{u}) - \frac{\gamma^2}{2}\e{h'_{\varepsilon}( \tilde B_{u})^2}}}{(| \tilde B_{u}| + \varepsilon)^{\alpha\gamma}} du < r^{\beta + \delta} \bigg| \mc F_t \right]\\
		&\leq \mb P\left[ e^{-(\alpha\gamma + \gamma^2) C} e^{\gamma\Omega_{\sqrt r}} < r^{\frac{\delta}{2}} \bigg| \mc F_t \right] + \mb P \left[ \int_{0}^{1 \wedge \tilde \tau} \frac{ e^{\gamma h'_{\varepsilon}( \tilde B_{u}) - \frac{\gamma^2}{2}\e{h'_{\varepsilon}( \tilde B_{u})^2}}}{(| \tilde B_{u}| + \varepsilon)^{\alpha\gamma}} du < r^{\frac{\delta}{2}} \right].
		\label{eq:split.prob}
	\end{align}
	The first term decays polynomially in $r$ as $r \to 0$ uniformly in $B_t$, by \cref{lem:gff.scaling}, and the second term decays polynomially in $r$ as $r \to 0$ by \cref{cor:lower.tail}. Therefore, looking back at \cref{eq:girsanov} again, there certainly exist some positive constants $D$ and $q$ such that
	\begin{align}
		\nonumber &\mb E \left[ \ind_{ \big\{ \int_{t}^{(t+r)\wedge \tau} e^{\gamma h_{\varepsilon\sqrt r}(B_s) - \frac{\gamma^2}{2}  \mb E [ h_{\varepsilon\sqrt r}(B_s)^2 ]}ds < r^{\beta+\delta} \big\}}  e^{\alpha h_{\varepsilon '}(B_t) - \frac{\alpha^2}{2} \e{h_{\varepsilon'}(B_t)^2}} \ind_{\left\{ \tau > t \right\}} \right] \leq \\
		\nonumber &\hspace{12em}\leq \e{ D r^{q} \ind_{ \left\{ \tau > t \right\} }} \\
		&\hspace{12em}= Dr^q \p{\tau > t}.
		\label{eq:ineq.conditional}
	\end{align}	
	When we integrate \cref{eq:ineq.conditional} over $t >0$, we find that
	\begin{equation}
		\mb E \Bigg[ \int_0^\tau \ind_{ \big\{ \int_{t}^{(t+r)\wedge \tau} e^{\gamma h_{\varepsilon\sqrt r}(B_s) - \frac{\gamma^2}{2}  \mb E [ h_{\varepsilon\sqrt r}(B_s)^2 ]}ds < r^{\beta+\delta} \big\}}  e^{\alpha h_{\varepsilon '}(B_t) - \frac{\alpha^2}{2} \e{h_{\varepsilon'}(B_t)^2}} dt \Bigg] \leq 
		\e{\tau} D r^{q}.
		\label{eq:ineq.single}
	\end{equation}

	Proposition 2.8 of \cite{garban2013liouville} tells us that, almost surely in $B$ and $h$, the measure defined by $\mu^{\varepsilon'}_\alpha(dt) = e^{\alpha h_{\varepsilon '}(B_t) - \frac{\alpha^2}{2} \e{h_{\varepsilon'}(B_t)^2}} dt$ converges weakly to the measure we have called $\mu_\alpha$. Therefore, as the set in the indicator function is open, we may use the portmanteau lemma and Fatou's lemma to see that
	\begin{align*}
		\mb E & \left[ \int_0^\tau \ind_{ \big\{ \int_{t}^{(t+r)\wedge \tau} e^{\gamma h_{\varepsilon\sqrt r}(B_s) - \frac{\gamma^2}{2}  \mb E [ h_{\varepsilon\sqrt r}(B_s)^2 ]}ds < r^{\beta+\delta} \big\}}  \mu_\alpha(dt) \right] \leq \\
		&\leq \mb E \Bigg[ \liminf_{\varepsilon' \to 0} \int_0^\tau \ind_{ \big\{ \int_{t}^{(t+r)\wedge \tau} e^{\gamma h_{\varepsilon\sqrt r}(B_s) - \frac{\gamma^2}{2}  \mb E [ h_{\varepsilon\sqrt r}(B_s)^2 ]}ds < r^{\beta+\delta} \big\}}  e^{\alpha h_{\varepsilon '}(B_t) - \frac{\alpha^2}{2} \e{h_{\varepsilon'}(B_t)^2}} dt \Bigg] \\
		&\leq \liminf_{\varepsilon' \to 0} \mb E \Bigg[ \int_0^\tau \ind_{ \big\{ \int_{t}^{(t+r)\wedge \tau} e^{\gamma h_{\varepsilon\sqrt r}(B_s) - \frac{\gamma^2}{2}  \mb E [ h_{\varepsilon\sqrt r}(B_s)^2 ]}ds < r^{\beta+\delta} \big\}}  e^{\alpha h_{\varepsilon '}(B_t) - \frac{\alpha^2}{2} \e{h_{\varepsilon'}(B_t)^2}} dt \Bigg] 
	\end{align*}
	Since $D$ and $q$ from \cref{eq:ineq.single} are independent of $\varepsilon'$, we can therefore see that
	\begin{equation*}
		\mb E \left[ \int_0^\tau \ind_{ \big\{ \int_{t}^{(t+r)\wedge \tau} e^{\gamma h_{\varepsilon\sqrt r}(B_s) - \frac{\gamma^2}{2}  \mb E [ h_{\varepsilon\sqrt r}(B_s)^2 ]}ds < r^{\beta+\delta} \big\}}  \mu_\alpha(dt) \right] \leq \e{\tau} D r^{q}.
	\end{equation*}
	We then use Fatou's lemma twice to conclude
	\begin{align*}
		\e{\mu_\alpha\left( \left\{ t\in[0,\tau] \, : \right.\right.\right. & \left.\left.\left. \! \mu_\gamma([t,(t+r)\wedge \tau]) < r^{\beta+\delta} \right\} \right) }= \\
		&= \e{\int_0^\tau \ind_{ \big\{ \mu_\gamma([t,(t+r)\wedge \tau]) < r^{\beta+\delta} \big\} } \mu_\alpha(dt)} \\
		&= \mb E \left[ \int_0^\tau \liminf_{\varepsilon\to0}\ind_{ \big\{ \int_{t}^{(t+r)\wedge\tau} e^{\gamma h_{\varepsilon\sqrt r}(B_s) - \frac{\gamma^2}{2}  \mb E [ h_{\varepsilon\sqrt r}(B_s)^2 ]}ds < r^{\beta+\delta} \big\}}  \mu_\alpha(dt) \right]\\
		&\leq \liminf_{\varepsilon \to 0} \mb E \left[ \int_0^\tau \ind_{ \big\{ \int_{t}^{(t+r)\wedge \tau} e^{\gamma h_{\varepsilon\sqrt r}(B_s) - \frac{\gamma^2}{2}  \mb E [ h_{\varepsilon\sqrt r}(B_s)^2 ]}ds < r^{\beta+\delta} \big\}}  \mu_\alpha(dt) \right]\\
		&\leq \e{\tau}D r^{q},
	\end{align*}
	and we are done, since $\e{\tau} < \infty$ (as it has exponentially decaying tails).
\end{proof}

\begin{proof}[Proof of \cref{prop:lower}]
	Using \cref{lem:size.bound} (taking $E = 2^{\beta+\delta}$) and Markov's inequality, we can bound the upper tail of the $\mu_\alpha$-measure of the set of times when $\mu_\gamma$ decays unusually fast by
	\begin{align*}
		\mb P & \left[ \mu_\alpha\left( \left\{ t \in [0,\tau] \, : \, \mu_\gamma([t,(t+r)\wedge \tau]) < 2^{\beta+\delta}r^{\beta+\delta} \right\} \right) \geq r^{q/2} \right] \leq \\
		&\hspace{4em}\leq r^{-q/2} \mb E \left[ \mu_\alpha\left( \left\{ t \in [0,\tau] \, : \, \mu_\gamma([t,(t+r)\wedge\tau]) < 2^{\beta+\delta}r^{\beta+\delta} \right\} \right) \right]\\
		&\hspace{4em}\leq D r^{q/2}.
	\end{align*}
	So, taking a sequence of scales $r_n = 2^{-n}$, we see that the events
	\begin{equation*}
		\left\{ \mu_\alpha\left( \left\{ t\in[0,\tau] \, : \, \mu_\gamma([t,(t+r_n)\wedge\tau]) < 2^{\beta+\delta}r_n^{\beta+\delta} \right\} \right) \geq r_{n}^{q/2} \right\}_{n \in \mb N}
	\end{equation*}
	occur only finitely often almost surely, by Borel-Cantelli. Therefore, for all $\Delta > 0$ (which may be random and depend on $\mu_\alpha([0,\tau])$), we can find a random but almost surely finite $N \in \mb N$ such that
	\begin{equation*}
		\mu_\alpha\left( \bigcup_{n\geq N} \left\{ t\in[0,\tau] \, : \, \mu_\gamma([t,(t+r_n)\wedge\tau]) < 2^{\beta+\delta}r_n^{\beta+\delta} \right\} \right) \leq \sum_{n \geq N} 2^{-qn/2} \leq \Delta,
	\end{equation*}
	and hence
	\begin{equation}
		\mu_\alpha\left( \bigcap_{n\geq N} \left\{ t\in[0,\tau] \, : \, \mu_\gamma([t,(t+r_n)\wedge\tau]) \geq 2^{\beta+\delta}r_n^{\beta+\delta} \right\} \right) \geq \mu_\alpha([0,\tau]) - \Delta.
		\label{eq:discrete.radii}
	\end{equation}
	We now need to infer the result for all $r \in (0,2^{-N})$ from the discrete set of radii we have it for in \cref{eq:discrete.radii}. So, let $t \in \bigcap_{n\geq N} \left\{ t\in[0,\tau] \, : \, \mu_\gamma([t,(t+r_n)\wedge\tau]) \geq r_n^{\beta+\delta} \right\}$, take $r \in (0, 2^{-N})$, and suppose $n$ is such that $r_{n+1} < r \leq r_n$. Then
	\begin{equation*}
		\mu_\alpha([t,(t+r)\wedge\tau]) \geq \mu_\alpha([t,(t+r_n)\wedge\tau]) \geq 2^{\beta+\delta} r^{\beta+\delta}_{n+1} = r_n^{\beta+\delta} \geq r^{\beta+\delta},
	\end{equation*}
	which implies that the discrete radii event is a subset of the continuous radii event:
	\begin{align*}
		\bigcap_{n\geq N} & \left\{ t\in[0,\tau] \, : \, \mu_\gamma([t,(t+r_n)\wedge\tau]) \geq 2^{\beta+\delta}r_n^{\beta+\delta} \right\} \subset \\
		& \hspace{4em} \subset \left\{ t \in [0,\tau] \, : \, \mu_\gamma([t,(t+r)\wedge\tau]) \geq r^{\beta + \delta} \quad \forall r \in [0,2^{-N}) \right\} = L^N_\gamma,
	\end{align*}
	and so we can conclude that
	\begin{equation*}
		\mu_\alpha(L^N_\gamma) \geq \mu_\alpha([0,\tau]) - \Delta.
	\end{equation*}
\end{proof}

We now state and prove a result which is essentially a ``converse'' to \cref{prop:lower}.

\begin{proposition}
	Fix $\delta > 0$, and let $\beta = 1 - \frac{\alpha\gamma}{2} + \frac{\gamma^2}{4}$. Define the set of times
	\begin{equation*}
		U^N_\gamma = \left\{ t\in [0,\tau] \, : \, \mu_\gamma([t,(t+r)\wedge \tau]) \leq r^{\beta-\delta} \quad \forall r \in [0,2^{-N})  \right\} 
	\end{equation*}
	Then for all $\Delta >0$, which may be random and depend on $\mu_\alpha([0,\tau])$, there exists some random but almost surely finite $N \in \mb N$ such that
	\begin{equation*}
		\mu_\alpha\left( U^N_\gamma \right) \geq \mu_\alpha\left( [0,\tau] \right) - \Delta.
	\end{equation*}
	\label{prop:upper}
\end{proposition}
This proposition is essentially saying that around $\alpha$-thick points, the map function $F_\gamma$ is $(\beta-\delta)$-H\"older continuous. To prove it, we need a lemma that is the equivalent of \cref{lem:size.bound}.

\begin{lemma}
	Fix $\delta>0$ and let $\beta = 1 - \frac{\alpha\gamma}{2} + \frac{\gamma^2}{4}$, and let $E>0$ be some positive constant. Then there exist two constants $D>0$ and $\eta>0$ such that
	\begin{equation*}
		\e{\mu_\alpha\left( \left\{ t\in[0,\tau] \, : \, \mu_\gamma([t,(t+r)\wedge \tau]) > E r^{\beta-\delta} \right\} \right) } \leq D r^{\eta}.
	\end{equation*}
	for all $r > 0$. 
	\label{lem:size.bound.2}
\end{lemma}

The introduction of long range correlations by Brownian motion is much more apparent in this proof than the proof of \cref{lem:size.bound}. Instead of re-scaling time by a factor of $r$ and space by a factor of $\sqrt r$ as we did previously, we will need to allow a bit of extra wiggle room. This is essentially due to the modulus of continuity of Brownian motion around time $r \downarrow 0$ being $\sqrt{2r \log \frac{1}{r}}$; we need a slightly lower power of $r$ to account for the $\log$ correction. We will introduce the radius $R$, which we will use as our scaling radius, and calculate what it needs to be closer to the end of the proof. 

\begin{proof}
	Again, we will prove this only in the case that $E = 1$. Let $r, R > 0$, and fix $\varepsilon, \varepsilon' >0 $ so that $\varepsilon' < \varepsilon \sqrt R$. Using Girsanov's change of measure theorem and \cref{lem:cov.bound} as we did in \cref{lem:size.bound}, we see that
		\begin{align}
			\nonumber \mb E_B \mb E_h \Big [ & \ind_{ \big\{ \int_{t}^{(t+r)\wedge \tau} e^{\gamma h_{\varepsilon\sqrt r}(B_s) - \frac{\gamma^2}{2}  \mb E [ h_{\varepsilon\sqrt r}(B_s)^2 ]}ds > r^{\beta-\delta} \big\}} e^{\alpha h_{\varepsilon '}(B_t) - \frac{\alpha^2}{2} \e{h_{\varepsilon'}(B_t)^2}} \ind_{ \left\{ \tau > t \right\} } \Big ] = \\[1em]
			&\hspace{4em} = \mb E_B \left[ \ind_{ \left\{ \tau > t \right\} } \mb P \left[ e^{\alpha\gamma C}\int_{t}^{(t+r)\wedge \tau} \frac{e^{\gamma h_{\varepsilon\sqrt{R}}(B_s) - \frac{\gamma^2}{2}\mb E\left[ h_{\varepsilon\sqrt{R}}(B_s)^2 \right]}}{\left( |B_s - B_t| + \varepsilon\sqrt{R} \right)^{\alpha\gamma}}ds > r^{\beta -\delta} \bigg| \mc F_{t}\right] \right],		
		\label{eq:girsanov.2}
	\end{align}
	where $\mc F_t = \sigma(B_s; \; s \leq t)$ is the natural filtration for $B$. Now, consider the integral in \cref{eq:girsanov.2}. We are looking for upper bounds on it, this time, to find an upper bound on the probability in \cref{eq:girsanov.2}. First of all, we apply a simple change of time, first $s \mapsto s-t$ and then $s = Ru$ to see that
	\begin{align*}
		\int_{t}^{(t+r)\wedge \tau} \frac{e^{\gamma h_{\varepsilon\sqrt{R}}(B_s) - \frac{\gamma^2}{2}\mb E\left[ h_{\varepsilon\sqrt{R}}(B_s)^2 \right]}}{\left( |B_s - B_t| + \varepsilon\sqrt{R} \right)^{\alpha\gamma}}ds 
		&= \int_{0}^{r \wedge (\tau-r)} \frac{e^{\gamma h_{\varepsilon\sqrt{R}}(B_{t+s}) - \frac{\gamma^2}{2}\mb E\left[ h_{\varepsilon\sqrt{R}}(B_{t+s})^2 \right]}}{\left( |B_{t+s} - B_t| + \varepsilon\sqrt{R} \right)^{\alpha\gamma}}ds	\\
		&= R\int_{0}^{\frac{r}{R} \wedge \frac{\tau-r}{R}} \frac{e^{\gamma h_{\varepsilon\sqrt{R}}(B_{t+Ru}) - \frac{\gamma^2}{2}\mb E\left[ h_{\varepsilon\sqrt{R}}(B_{t+Ru})^2 \right]}}{\left( |B_{t + Ru} - B_t| + \varepsilon\sqrt{R} \right)^{\alpha\gamma}}du	\\
		&\overset{d}{=}R \int_{0}^{\frac{r}{R} \wedge \tau'} \frac{e^{\gamma h_{\varepsilon\sqrt{R}}(\sqrt{R} \tilde B_u + B_{t}) - \frac{\gamma^2}{2}\mb E\left[ h_{\varepsilon\sqrt{R}}(\sqrt{R} \tilde B_u + B_{t+Ru})^2 \right]}}{\left( |\sqrt{R} \tilde B_u| + \varepsilon\sqrt{R} \right)^{\alpha\gamma}}du \\
		&= R^{1 - \frac{\alpha\gamma}{2}} \int_{0}^{\frac{r}{R} \wedge \tau'} \frac{e^{\gamma h_{\varepsilon\sqrt{R}}(\sqrt{R} \tilde B_u + B_{t}) - \frac{\gamma^2}{2}\mb E\left[ h_{\varepsilon\sqrt{R}}(\sqrt{R}\tilde B_u + B_{t+Ru})^2 \right]}}{\left( |\tilde B_u| + \varepsilon \right)^{\alpha\gamma}}du
	\end{align*}
	where $\tilde B$ is an independent Brownian motion started at zero, and 
	\begin{equation*}
		\tau' = \inf\left\{ u >0 \; : \; |\sqrt{R} \tilde B_u + B_t| = \frac{1}{2} \right\}. 
	\end{equation*}
	The equality in distribution holds $\mb P_h$-almost surely.  
	
	We now want to use the scaling properties of the field $h$, from \cref{lem:gff.scaling}. So, as before, we use the Markov property of the GFF to write $h = h^{har} + \tilde h$, where $h^{har}$ is the harmonic projection of $h$ onto the disc of radius $2\sqrt R$, centred at $B_t$, and $\tilde h$ has the law of a zero-boundary GFF on the disc of radius $2\sqrt R$, centred at $B_t$. If we write $\Omega_{\sqrt R} = \sup_{z \in B(B_t, \sqrt R)} h^{har}(z)$, we know that $h \leq \Omega_{\sqrt R} + \tilde h$ inside the disc $B(B_t, \sqrt{R})$. In order to use \cref{lem:gff.scaling}, we need to make sure that the $\tilde B$ does not move far from its starting point.	 So, let $\tilde \tau$ be the exit time of $\tilde B$ from the unit disc. Then, on the event $\left\{ \tilde \tau > \frac{r}{R}\right\}$ we can see that
	\begin{align}
		\nonumber R^{1 - \frac{\alpha\gamma}{2}}&\int_{0}^{\frac{r}{R}\wedge \tau'} \frac{e^{\gamma h_{\varepsilon\sqrt R}(\sqrt R \tilde B_{u}) - \frac{\gamma^2}{2}\mb E\left[ h_{\varepsilon\sqrt R}(\sqrt R \tilde B_{u})^2 \right]}}{\left( |\tilde B_{u}| + \varepsilon \right)^{\alpha\gamma}}du \leq \\
		\nonumber &\hspace{4em}\leq R^{1 - \frac{\alpha\gamma}{2}} \int_0^{\frac{r}{R}\wedge \tau'} \frac{e^{\gamma h_{\varepsilon\sqrt R}(\sqrt R \tilde B_u) - \frac{\gamma^2}{2}\left(-\log(\varepsilon\sqrt R) - C\right)}}{ (|\tilde B_u| + \varepsilon)^{\alpha\gamma} }du \\
		\nonumber &\hspace{4em}= R^{1 - \frac{\alpha\gamma}{2} + \frac{\gamma^2}{4}} e^{\frac{\gamma^2}{2}C} \int_0^{\frac{r}{R}\wedge \tau'} \frac{e^{\gamma h_{\varepsilon\sqrt R}(\sqrt R \tilde B_u) + \frac{\gamma^2}{2}\log(\varepsilon\sqrt R)}}{ (|\tilde B_u| + \varepsilon)^{\alpha\gamma} }du \\
		\nonumber &\hspace{4em}\leq R^{1 - \frac{\alpha\gamma}{2} + \frac{\gamma^2}{4}} e^{\frac{\gamma^2}{2}C} e^{\gamma \Omega_{\sqrt R}} \int_0^{\frac{r}{R}\wedge \tau'} \frac{e^{\gamma \tilde h_{\varepsilon\sqrt R}(\sqrt R \tilde B_u) + \frac{\gamma^2}{2}\log(\varepsilon\sqrt R)}}{ (|\tilde B_u| + \varepsilon)^{\alpha\gamma} }du \\
		&\hspace{4em}\leq e^{\gamma^2 C} R^{\beta}e^{\gamma \Omega_{\sqrt R}}\int_0^{\frac{r}{R}\wedge \tau'}\frac{e^{\gamma h'_\varepsilon(\tilde B_u) - \frac{\gamma^2}{2}\mb E\left[ h'_\varepsilon(\tilde B_u)^2 \right]}}{\left(|\tilde B_u| + \varepsilon\right)^{\alpha\gamma}}du,
		\label{eq:pen}
	\end{align}
	where $h'$ is a zero boundary GFF on the unit disc. Therefore, we can use the right hand side of \cref{eq:pen} to bound the probability in \cref{eq:girsanov.2} by
	\begin{align}
  	\nonumber\mb P& \Bigg[ e^{\alpha\gamma C} \int_{t}^{(t+r)\wedge \tau} \frac{e^{\gamma h_{\varepsilon\sqrt R}(B_s) - \frac{\gamma^2}{2}\mb E\left[ h_{\varepsilon\sqrt R}(B_s)^2 \right]}}{\left( |B_s - B_t| + \varepsilon\sqrt{R} \right)^{\alpha\gamma}}ds > r^{\beta -\delta} \bigg| \mc F_t\Bigg]\\
		\nonumber&\leq \mb P \left[ e^{(\alpha\gamma + \gamma^2) C} R^{\beta}e^{\gamma \Omega_{\sqrt R}}\int_0^{\frac{r}{R}\wedge \tau'}\frac{e^{\gamma h'_\varepsilon(\tilde B_u) - \frac{\gamma^2}{2}\mb E\left[ h'_\varepsilon(\tilde B_u)^2 \right]}}{(|\tilde B_u| + \varepsilon)}du > r^{\beta -\delta} ; \tilde \tau > \frac{r}{R} \bigg| \mc F_t \right] + 
		%
		\mb P \left[ \tilde{\tau} < \frac{r}{R} \right]\\
		\nonumber&\leq \mb P \left[ e^{(\alpha\gamma + \gamma^2) C} R^{\beta}e^{\gamma \Omega_{\sqrt R}}\int_0^{\tilde \tau \wedge \tau'}\frac{e^{\gamma h'_\varepsilon(\tilde B_u) - \frac{\gamma^2}{2}\mb E\left[ h'_\varepsilon(\tilde B_u)^2 \right]}}{(|\tilde B_u| + \varepsilon)}du >  r^{\beta -\delta} \bigg| \mc F_t \right] + \mb P \left[ \tilde \tau < \frac{r}{R} \right]\\
		\nonumber&\leq \p{ e^{(\alpha\gamma + \gamma^2) C}e^{\gamma\Omega_{\sqrt R}} > \left(\frac{r}{R}\right)^{\frac{\beta}{2}}r^{-\frac{\delta}{2}} \bigg| \mc F_t } + \\
		&\hspace{4em}+ \p{ \int_0^{\tilde \tau}\frac{e^{\gamma h_\varepsilon(\tilde B_u) - \frac{\gamma^2}{2}\mb E\left[ h_\varepsilon(\tilde B_u)^2 \right]}}{(|\tilde B_u| + \varepsilon)}du > \left(\frac{r}{R}\right)^{\frac{\beta}{2}}r^{-\frac{\delta}{2}} } + \p{ \tilde \tau < \frac{r}{R} }.
		\label{eq:polynomial}
	\end{align}
	We are now in a position to see what choice we should make for the radius $R$. We want $\frac{r}{R} \to 0$ as $r \to 0$ polynomially in $r$, so that the third term in \cref{eq:polynomial} decays polynomially. We also want $\left( \frac{r}{R} \right)^{\frac{\beta}{2}}r^{-\frac{\delta}{2}}$ to converge to infinity, polynomially in $r$, so that the other terms in \cref{eq:polynomial} also decay polynomially: see below. The choice $R = r^{1 - \frac{\delta^2}{2}}$ works for $\delta$ small enough, since then we certainly have $\frac{r}{R} = r^{\delta^2} \to 0$, and also $\left( \frac{r}{R} \right)^{\frac{\beta}{2}}r^{-\frac{\delta}{2}} = r^{\frac{\beta}{2}\delta^2 - \frac{\delta}{2}} \rightarrow \infty$. The exponent $\frac{\beta}{2}\delta^2 - \frac{1}{2}\delta$ is negative for $\delta$ small enough, and so we have the desired properties. With this choice of $R$, the first term on the right hand side of \cref{eq:polynomial} decays polynomially by \cref{lem:gff.scaling}, the second term decays polynomially by \cref{cor:upper.tail}. We can bound the third term above by
	\begin{equation*}
		\p{\tilde \tau < \frac{r}{R}} \leq \p{T < \frac{r}{R}},
	\end{equation*}
	where $T$ is the exit time of a one dimensional Brownian motion from the interval $[-\frac{1}{\sqrt 2}, \frac{1}{\sqrt 2}]$. The stopping time $T$ has exponentially decaying tails, and so we can see that the third term decays polynomially as well. 
	
	Therefore, going all the way back to \cref{eq:girsanov.2}, there certainly exist some constants $D>0$ and $q>0$ such that, for $t > 0$, 
	\begin{equation*}
		\mb E \Big [ \ind_{ \big\{ \int_{t}^{(t+r)\wedge\tau} e^{\gamma h_{\varepsilon\sqrt r}(B_s) - \frac{\gamma^2}{2}  \mb E [ h_{\varepsilon\sqrt r}(B_s)^2 ]}ds > r^{\beta-\delta} \big\}}  e^{\alpha h_{\varepsilon '}(B_t) - \frac{\alpha^2}{2} \e{h_{\varepsilon'}(B_t)^2}} \ind_{ \left\{ \tau > t \right\} } \Big ] \leq D r^{\eta} \p{\tau > t}
	\end{equation*}
	Integrating over $t > 0$ gives
	\begin{equation}
		\mb E \Big [ \int_0^{\tau} \ind_{ \big\{ \int_{t}^{(t+r)\wedge\tau} e^{\gamma h_{\varepsilon\sqrt r}(B_s) - \frac{\gamma^2}{2}  \mb E [ h_{\varepsilon\sqrt r}(B_s)^2 ]}ds > r^{\beta-\delta} \big\}}  e^{\alpha h_{\varepsilon '}(B_t) - \frac{\alpha^2}{2} \e{h_{\varepsilon'}(B_t)^2}} dt \Big ] \leq D r^{\eta} \e{\tau},
		\label{eq:ineq.single.2}
	\end{equation}
	Now, note that \cref{eq:ineq.single.2} is almost identical to \cref{eq:ineq.single}. We use the same arguments to let $\varepsilon'$ and $\varepsilon$ converge to zero, and conclude that  
	\begin{equation*}
		\e{\mu_\alpha\left( \left\{ t\in[r,\tau] \, : \, \mu_\gamma([(t)\vee 0,(t+r)\wedge \tau]) > E r^{\beta-\delta} \right\} \right) } \leq D r^{\eta}\e{\tau}.
	\end{equation*}
	which completes the proof, as $\e{\tau} < \infty$. 
\end{proof}

\begin{proof}[Proof of \cref{prop:upper}] \Cref{prop:upper} follows from \cref{lem:size.bound.2} in exactly the same way that \cref{prop:lower} followed from \cref{lem:size.bound}. 
	
\end{proof}

\subsection{Proof of \cref{thm:lower.bound} and \cref{thm:upper.bound}}\label{subs:proofs}

We now have all of the tools ready to prove \cref{thm:lower.bound}, which we re-state here in more detail. 

\begin{theorem}
	Let $B$ be the Brownian motion used to construct the LBM time changes $F_\alpha$ and $F_\gamma$. Call $T_\alpha = \left\{ t > 0 \, : \, B_t \in \mc T_\alpha \right\}$ the set of times that the Brownian motion $B$ is in an $\alpha$-thick point. Then
	\begin{equation*}
		\dim_H(F_\gamma(T_\alpha)) \geq \frac{1 - \frac{\alpha^2}{4}}{1 - \frac{\alpha\gamma}{2} + \frac{\gamma^2}{4}}
	\end{equation*}
  where, by the definition of the change of time, $F_\gamma(T_\alpha)$ is the set of times the $\gamma$-LBM is in $\alpha$-thick points.
	\label{thm:lower.bound.exact}
\end{theorem}

\begin{proof}
	We will in fact prove the lower bound for times only only for the stopping time $\tau$ when the Brownian motion leaves the disc of radius $\frac{1}{2}$. To that end, we will abuse notation slightly and re-define the set of times $T_\alpha$ as
	\begin{equation*}
		T_\alpha = \left\{ t \in [0,\tau] \, : \, B_t \in \mc T_\alpha \right\}. 
	\end{equation*}
	Because $T_\alpha \cap L^N_\gamma \subset T_\alpha$, where $L^N_\gamma$ is defined as in \cref{prop:lower}, we know that 
	\begin{equation*}
		\dim_H(F_\gamma(T_\alpha \cap L^N_\gamma)) \leq \dim_H(F_\gamma(T_\alpha)). 
	\end{equation*}
	But, because	$F_\gamma^{-1}$ is a $\frac{1}{\beta+\delta}$-H\"older-like function, in the sense of \cref{defn:holder.like}, on intervals starting at times in the image $F_\gamma(L^N_\gamma)$, \cref{prop:holder} implies that
	\begin{equation}
		\dim_H(T_\alpha \cap L^N_\gamma) \leq (\beta+\delta) \dim_H(F_\gamma(T_\alpha \cap L^N_\gamma)),
		\label{eq:lower}
	\end{equation}
	and so to get a lower bound on $\dim_H(T_\alpha)$, we want to find a lower bound for $\dim_H(T_\alpha \cap L^N_\gamma)$. 
	We now use \cref{prop:lower,prop:upper} to see that we can take $N$ large enough to ensure that $T_\alpha \cap L^N_\gamma \cap U^N_\alpha$ has positive $\mu_\alpha$-measure (taking $\gamma = \alpha$ in \cref{prop:upper}, and $\Delta = \frac{1}{4}\mu_\alpha([0,\tau])$ for example). Since we also know that $\mu_\alpha[0,\tau] < \infty$ almost surely, the measure $\mu_\alpha$ defines a mass distribution on the set $T_\alpha \cap L^N_\gamma \cap U^N_\alpha$, and by the definition of $U^N_\alpha$, we know that 
	\begin{equation*}
		\mu_\alpha([t,t+r]) \leq r^{1 - \frac{\alpha^2}{4} - \delta}
	\end{equation*}
	for all $r \in [0, 2^{-N})$ and $t \in T_\alpha \cap L^N_\gamma \cap U^N_\alpha$. So, by the mass distribution principle (Theorem 4.19 of \cite{morters2010brownian} for example), we find that
	\begin{equation*}
		\dim_H(T_\alpha \cap L^N_\gamma \cap U^N_\alpha) \geq 1 - \frac{\alpha^2}{4} - \delta.
	\end{equation*}
	Therefore we certainly have the bound $\dim_H(T_\alpha \cap L^N_\gamma) \geq 1 - \frac{\alpha^2}{4} - \delta$, which we can substitute into \cref{eq:lower} and re-arrange to find
	\begin{equation*}
		\dim_H(F_\gamma(T_\alpha)) \geq \frac{1 - \frac{\alpha^2}{4} - \delta}{1 - \frac{\alpha\gamma}{2} + \frac{\gamma^2}{4} + \delta}.
	\end{equation*}
	Since $\delta$ was arbitrary we can take the limit $\delta \to 0$, and we have shown the result.
\end{proof}

And now we re-state \cref{thm:upper.bound} and prove it:

\begin{corollary}
	Let $B$ be the Brownian motion used to construct the LBM time changes $F_\alpha$ and $F_\gamma$. Call $T_\alpha = \left\{ t \in [0,T] \, : \, B_t \in \mc T_\alpha \right\}$ the set of times that the Brownian motion $B$ is in an $\alpha$-thick point. Then
	\begin{equation*}
		\dim_H(F_\gamma(T_\alpha)) = \frac{1 - \frac{\alpha^2}{4}}{1 - \frac{\alpha\gamma}{2} + \frac{\gamma^2}{4}}
	\end{equation*}
  where, by the definition of the change of time, $F_\gamma(T_\alpha)$ is the set of times the $\gamma$-LBM is in $\alpha$-thick points.
\end{corollary}

\begin{proof}
	First, following from Theorem 1.4 of \cite{berestycki2013diffusion}, we introduce the sets
	\begin{align*}
		\mc T^-_\alpha = \left\{ z \in \mb C \; : \; \liminf_{\varepsilon \to 0} \frac{h_\varepsilon(z)}{\log \frac{1}{\varepsilon}} \geq \alpha \right\}, \\
		\mc T^+_\alpha = \left\{ z \in \mb C \; : \; \limsup_{\varepsilon \to 0} \frac{h_\varepsilon(z)}{\log \frac{1}{\varepsilon}} \leq \alpha \right\}
	\end{align*}
	We will call $T^-_\alpha = \left\{ t \in [0,\tau] \; : \; B_t \in \mc T^-_\alpha \right\}$, and similarly define $T^+_\alpha$ from $\mc T^+_\alpha$. We know, from \cite{berestycki2013diffusion}, that for $\alpha > \gamma$ we have the upper bound
	\begin{equation*}
		\dim_H (F_\gamma(T^-_\alpha)) \leq \frac{1 - \frac{\alpha^2}{4}}{1 - \frac{\alpha\gamma}{2} + \frac{\gamma^2}{4}},
	\end{equation*}
	and the same result holds when we have $\alpha < \gamma$ and we replace $T^-_\alpha$ with $T^+_\alpha$.

	Let us consider the case $\alpha > \gamma$. We know that $T_\alpha \subset T^-_\alpha$, and so we have
	\begin{equation*}
		\frac{1 - \frac{\alpha^2}{4}}{1 - \frac{\alpha\gamma}{2} + \frac{\gamma^2}{4}} \leq \dim_H(F_\gamma(T_\alpha)) \leq \dim_H(F_\gamma(T^-_\alpha)) \leq \frac{1 - \frac{\alpha^2}{4}}{1 - \frac{\alpha\gamma}{2} + \frac{\gamma^2}{4}},
	\end{equation*}
	showing us the equality. We can show equality in the case $\alpha < \gamma$ in the same way.
\end{proof}

\subsection{Proofs of Regularity Properties}\label{subs:regularity}

We can now state \cref{thm:regularity} again, and give the proof.

\begin{theorem}
	For $\mu_\alpha$-almost every $t \geq 0$, the change of time $F_\gamma$ has the following growth rate:
	\begin{equation}
		\lim_{r \rightarrow 0} \frac{\log |F_\gamma(t) - F_\gamma(t+r)|}{\log |r|} = 1 - \frac{\alpha\gamma}{2} + \frac{\gamma^2}{4},
		\label{eq:regularity}
	\end{equation}
	almost surely. 
	\label{thm:regularity.restated}
\end{theorem}

Before we start the proof, we would like to explain the intuition behind $\mu_\alpha$-almost every $t \in T_\alpha$. Suppose we have our GFF $h$, and the Brownian motion $B$ which is the path of our Liouville Brownian motion. We now run an $\alpha$-LBM, $Z^\alpha$, along the path $B$, using $h$ to calculate the time change $F_\alpha$. At some time $\mathbf t$, chosen uniformly at random from the lifetime of $Z^\alpha$, we inspect the point in the plane occupied by $Z^\alpha_{\mathbf t}$. Because an $\alpha$-LBM spends Lebesgue-almost all of its time in $\alpha$-thick points, we know that the point chosen by $Z^\alpha_{\mathbf t}$ is an $\alpha$-thick point, almost surely. We also know that it is on the path of the Brownian motion $B$. If we call the time $B$ passes through this point $t$, i.e.~$t = F^{-1}_\alpha(\mathbf t)$, we know that, around this time, the $\gamma$-time change, $F_\gamma$, has the regularity property given in \cref{eq:regularity}.

To prove \cref{thm:regularity.restated}, we will first prove it while taking the limit $r \to 0$, i.e.~as $r$ approaches $0$ from above. 

\begin{lemma}
	For $\mu_\alpha$-almost every $t \geq 0$, the change of time $F_\gamma$ has the following growth rate:
	\begin{equation*}
		\lim_{r \to 0} \frac{\log |F_\gamma(t) - F_\gamma(t+r)|}{\log r} = 1 - \frac{\alpha\gamma}{2} + \frac{\gamma^2}{4},
	\end{equation*}
	almost surely. 
	\label{lem:regularity.restated}
\end{lemma}

\begin{proof}
	Most of the work for this proof has been done in \cref{prop:lower,prop:upper}. Recall that for some arbitrary $\delta>0$ we defined 
	\begin{equation*}
		L^N_\gamma = \left\{ t\in [0,\tau] \, : \, \mu_\gamma([t,(t+r)\wedge \tau]) \geq r^{\beta+\delta} \quad \forall r \in [0,2^{-N})  \right\},
	\end{equation*}
	and
	\begin{equation*}
		U^N_\gamma = \left\{ t\in [0,\tau] \, : \, \mu_\gamma([t,(t+r)\wedge \tau]) \leq r^{\beta-\delta} \quad \forall r \in [0,2^{-N})  \right\}.
	\end{equation*}
	Now let us define
	\begin{equation*}
		L_\gamma = \bigcup_N L^N_\gamma = \left\{ t \in [0,\tau) \, : \, \limsup_{r \to 0}\frac{\log\mu_\gamma([t,t+r]))}{\log r} \leq \beta + \delta \right\},
	\end{equation*}
	and similarly define $U_\gamma = \bigcup_N U^N_\gamma$.

	We showed in \cref{prop:lower,prop:upper} that for any $\Delta > 0$, we could find $N$ large enough that 
	\begin{equation*}
		\mu_\alpha(L^N_\gamma) \geq \mu_\alpha([0,\tau]) - \Delta,
	\end{equation*}
	and
	\begin{equation*}
		\mu_\alpha(U^N_\gamma) \geq \mu_\alpha([0,\tau]) - \Delta.
	\end{equation*}
	Since $\Delta$ was arbitrary and $L^N_\gamma$, $U^N_\gamma$ are increasing sets, we find that
	\begin{equation*}
		\mu_\alpha(L_\gamma \cap U_\gamma) = \mu_\alpha([0,\tau]).
	\end{equation*}
	Because $\delta$ was arbitrary, and we defined $\mu_\gamma([t,t+r]) := F(t+r) - F(t)$, we have shown the result. 
\end{proof}

We now need a lemma which allows us to ``reverse time'' in some way, and extend the result from \cref{lem:regularity.restated} to the statement in \cref{thm:regularity.restated}. 

\begin{lemma}
	Let $B$ be a Brownian motion started at zero, and let $t > 0$. Define two stochastic processes, conditional on $B_t$, by setting
	\begin{equation*}
		W^+_s = B_{t+s}
	\end{equation*}
	for $s \geq 0$, and
	\begin{equation*}
		W^-_s = B_{t - s}
	\end{equation*}
	for $s \in [0,t]$. Now, let $\varepsilon < t$. Then, conditional on the event $\{\tau > t\}$ (where $\tau$ is the first exit time of $B$ from the disc of radius $\frac{1}{2}$), the laws of the restricted processes $(W^+_s)_{s \in [0,\varepsilon]}$ and $(W^-_s)_{s \in [0,\varepsilon]}$ are absolutely continuous with respect to each other.
	\label{lem:abs.cont}
\end{lemma}

\begin{proof}
	Conditional on $B_t = z$ and the event $\left\{ \tau > t \right\}$, the law of the Brownian motion $B$ is that of a Brownian bridge of duration $t$, joining the origin and $z$, conditioned to stay inside the disc of radius $\frac{1}{2}$, followed by an independent Brownian motion started at $z$. Because the event that the maximum modulus of this Brownian bridge is less than $\frac{1}{2}$ has positive probability, it does not affect the absolute continuity of measures. So for the rest of the proof, we may ignore the fact that we are conditioning on that event.
	
	By reversibility of Brownian bridges, the process $W^-$ has the law of a Brownian bridge of duration $t$, connecting $z$ and the origin. And, as stated above, $W^+$ has the law of a Brownian motion started at $z$. So, by (6.28) of \cite{karatzas1991brownian} (or, slightly more explicitly, Lemma 3.1 of \cite{berestycki2014kpz}), we see that the laws of a Brownian bridge of duration $t$ and a Brownian motion, with a common starting point, are absolutely continuous with respect to each other on intervals shorter than $t$. 
\end{proof}

\begin{proof}[Proof of \cref{thm:regularity.restated}]
	Let $T$ be an exponential random variable with mean $1$, independent of the GFF $h$ and the Brownian motion $B$. Recall that the measure $\mu_\alpha$ is defined by 
	\begin{equation*}
		\mu_\alpha([a,b]) = F_\alpha(b) - F_\alpha(a),
	\end{equation*}
	which can also be written as $\mu_\alpha = \mathcal{L}eb \circ F_\alpha$. Now, because the law of $T$ is absolutely continuous with respect to Lebesgue measure and $F_\alpha$ is a bijection, the law of $F^{-1}_\alpha(T)$ is absolutely continuous with respect to $\mu_\alpha$. Therefore, by \cref{lem:regularity.restated}, we see that
	\begin{equation*}
		F^{-1}_\alpha(T) \in \left\{ t > 0 \: : \; \lim_{r \to 0} \frac{\log | F_\gamma(t) - F_\gamma(t + r) |}{\log r} = \beta \right\},
	\end{equation*}
	almost surely. It therefore follows from \cref{lem:abs.cont} that we also have
	\begin{equation*}
		F^{-1}_\alpha(T) \in \left\{ t > 0 \: : \; \lim_{r \to 0} \frac{\log | F_\gamma(t) - F_\gamma(t - r) |}{\log r} = \beta \right\},
	\end{equation*}
	almost surely. Finally, by absolute continuity of the law of $F^{-1}_\alpha(T)$ and the measure $\mu_\alpha$ again, we deduce that for $\mu_\alpha$-almost every $t$ we have
	\begin{equation*}
		\lim_{r \rightarrow 0} \frac{\log | F_\gamma(t) - F_\gamma(t + r) |}{\log |r|} = \beta
	\end{equation*}
	almost surely, completing the proof. 
\end{proof}

We can use the regularity property of $F_\gamma$ from \cref{thm:regularity.restated} that we have just shown to find a bound on the growth rate of LBM around thick points of different levels. We first prove a lemma about the growth rate of LBM given a lot of control on how we choose the time we consider. We will then extend that to the more general statement given in \cref{cor:regularity}.

\begin{lemma}\label{lem:lbm.scale}
	Let $t \geq 0$ be such that
	\begin{equation*}
		\lim_{r \rightarrow 0} \frac{\log |F_\gamma(t) - F_\gamma(t+r)|}{\log |r|} = 1 - \frac{\alpha\gamma}{2} + \frac{\gamma^2}{4}.
	\end{equation*}
	Then
	\begin{equation*}
		\limsup_{u \rightarrow 0}\frac{\log |Z_{F_\gamma(t)} - Z_{F_\gamma(t)+u}|}{\log |u|} = \frac{1}{2 - \alpha\gamma+ \frac{\gamma^2}{2}}
	\end{equation*}
	almost surely.
\end{lemma}

\begin{note}
	In \cref{lem:lbm.scale}, we have let $r \rightarrow 0$ and $u \rightarrow 0$ from above and below. In the proof of \cref{cor:m.alpha}, only the result as $r \to 0$ and $u \to 0$ are used, but the distinction is important for the proof of \cref{cor:differentiability.restated}.
\end{note}

\begin{proof}
	Let $\delta > 0$. Then by L\'evy's modulus of continuity of Brownian motion, we know that, almost surely, there exists some $S < \infty$ such that
	\begin{equation}
		|B_t - B_{t + s}| \leq s^{\frac{1}{2} - \delta}
		\label{eq:bm}
	\end{equation}
	for all $s \in [-S, S]$, and for all $\varepsilon > 0$ there exists some $s \in [-\varepsilon,\varepsilon]$ such that
	\begin{equation}
		|B_t - B_{t+s}| \geq s^{\frac{1}{2}+\delta}.
		\label{eq:bm.lower}
	\end{equation}

	Now, let us write $\beta = 1 - \frac{\alpha\gamma}{2} + \frac{\gamma^2}{4}$. Then by assumption, there exists some $R < \infty$ such that
	\begin{equation*}
		r^{\beta + \delta} \leq F_\gamma(t+r) - F_\gamma(t) \leq r^{\beta-\delta}
	\end{equation*}
	for all $r \in [-R,R]$. Since $F_\gamma^{-1}$ is well defined, this in turn implies that, for all $u$ with $|u|$ small enough, 
	\begin{equation}
		u^{\frac{1}{\beta-\delta}} \leq F_\gamma^{-1}( F_\gamma(t) + u) - t \leq u^{\frac{1}{\beta+\delta}}.
		\label{eq:lbm}
	\end{equation}

	Recalling the definition $Z_{t}^\gamma = B_{F_\gamma^{-1}(t)}$ and combining \cref{eq:bm} and \cref{eq:lbm} shows us that
	\begin{equation}
		|Z_{F_\gamma(t)} - Z_{F_\gamma(t)+u} | \leq \left( |u|^{\frac{1}{\beta+\delta}} \right)^{\frac{1}{2} - \delta} 
		\label{eq:Z.bound}
	\end{equation}
	for all $|u|$ small enough. Furthermore, combining \cref{eq:bm.lower} and \cref{eq:lbm} shows us that, for any $\varepsilon' > 0$ there exists some $u \in [-\varepsilon', \varepsilon']$ such that
	\begin{equation*}
	|Z_{F_\gamma(t)} - Z_{F_\gamma(t)+u} | \geq \left( |u|^{\frac{1}{\beta-\delta}} \right)^{\frac{1}{2} + \delta}.
	\end{equation*}
	Taking logs then implies that
	\begin{equation*}
		\frac{1 - \delta}{2(\beta+\delta)}  \leq \limsup_{u \rightarrow 0} \frac{\log |Z_{F_\gamma(t)} - Z_{F_\gamma(t)+u}|}{\log |u|} \leq \frac{1+\delta}{2(\beta-\delta)}
	\end{equation*}
	almost surely. Therefore, letting $\delta \to 0$ along a countable sequence shows us that the limsup equals $\frac{1}{2\beta}$ almost surely, as claimed. 
\end{proof}

We can now use the results from \cref{thm:regularity.restated} and \cref{lem:lbm.scale} to prove \cref{cor:regularity}, which we restate here.

\begin{corollary}
	Suppose that the starting point of a $\gamma$-Liouville Brownian motion is chosen according to $M_\alpha$, i.e.~$Z^\gamma_0 \sim M_\alpha$. Then
	\begin{equation*}
		\limsup_{t \to 0}\frac{\log |Z^\gamma_t|}{\log t} = \frac{1}{2 - \alpha\gamma+ \frac{\gamma^2}{2}}
	\end{equation*}
	almost surely.
	\label{cor:m.alpha}
\end{corollary}

\begin{proof}
	Let $T$ be an exponential random variable with mean 1, which is independent of the GFF $h$ and the Brownian motion $B$. 
	
	By the same reasoning as that used in the proof of \cref{thm:regularity.restated}, we see that
	\begin{equation*}
		F^{-1}_\alpha(T) \in \left\{ t \geq 0 \: : \; \lim_{r \to 0} \frac{\log | F_\gamma(t) - F_\gamma(t + r) |}{\log r} = \beta \right\},
	\end{equation*}
	almost surely. (If $T > \sup_t F_\alpha(t)$, we set $F^{-1}_\alpha(T) = \emptyset$, and claim that the equality below holds, vacuously.) Therefore if we write $T' = F_\gamma(F^{-1}_\alpha(T))$, \cref{lem:lbm.scale} tells us that
	\begin{equation*}
		\limsup_{u \to 0} \frac{\log|Z_{T'}^\gamma - Z^\gamma_{T' + u}|}{\log u} = \frac{1}{2\beta}.
	\end{equation*}

	Let $\mc H$ be the sigma algebra generated by the GFF $h$, i.e.
	\begin{equation*}
		\mc H = \sigma\left( \dipd{h}{f} \; : \; f \in H^1_0(\mc D) \right). 
	\end{equation*}
	Now consider the filtration defined by 
	\begin{align*}
		\mc G_t &= \sigma(Z^\gamma_s \; : \; s < t) \vee \mc H\\
		&= \sigma(B_s \; : \; s < F_\gamma^{-1}(t)) \vee \mc H.
	\end{align*}
	The process $Z^\gamma$ is certainly $\mc G_t$-adapted, and $T'$ is a $\mc G_t$-stopping time since
	\begin{equation*}
		\left\{ T' > t \right\} = \left\{ F_\gamma(F_\alpha^{-1}(T)) > t \right\} = \left\{ T > F_\alpha(F_\gamma^{-1}(t)) \right\},
	\end{equation*}
	and
	\begin{equation*}
		F_\alpha(F_\gamma^{-1}(t)) = \lim_{\varepsilon \to 0} \int_0^{F_\gamma^{-1}(t)} e^{\alpha h_\varepsilon(B_s) - \frac{\alpha^2}{2}\e{h_\varepsilon(B_s)^2}}ds
	\end{equation*}
	is $\mc G_t$-measurable. 

	We can therefore use the strong Markov property of $Z^\gamma$ to deduce that
	\begin{equation}
		\limsup_{t \to 0}\frac{\log |Z^\gamma_t|}{\log t} = \frac{1}{2\beta}
		\label{eq:limsup.lbm}
	\end{equation}
	whenever $Z_0^\gamma$ is chosen according to $\mc P^\alpha_T$, the law of $Z^\alpha_T$. 

	From Theorem 2.5 in \cite{garban2013heat}, we know that, for a fixed $t \geq 0$,  the law of $Z^\alpha_t$ is absolutely continuous with respect to the Liouville measure $M^\alpha$, with Radon-Nikodym derivative 
	\begin{equation*}
		\frac{d\mc P^\alpha_t}{d M^\alpha}(y) = p_t^\alpha(0,y) \geq 0.
	\end{equation*}
	We can therefore write
	\begin{equation*}
		\frac{d \mc P^\alpha_T}{d M^\alpha}(y) = \int_0^\infty e^{-t} p_t^\alpha(0,y)dt.
	\end{equation*}
	Theorem 2.5 of \cite{garban2013heat} also implies that, for $M^\alpha$-almost every $y\in \mc D$, the transition density $p_t^\alpha(0,y)$ is strictly positive for all $t$ in a measurable set with positive Lebesgue measure. (This fact was noted in an earlier version of their paper.) But that implies that
	\begin{equation*}
	\frac{d \mc P^\alpha_T}{d M^\alpha}(y)> 0	
	\end{equation*}
	for $M^\alpha$-almost every $y \in \mc D$, i.e.~the Liouville measure $M^\alpha$ and $\mc P^\alpha_T$ are absolutely continuous with respect to each other. Therefore, since \cref{eq:limsup.lbm} holds almost surely whenever $Z_0^\gamma$ was chosen according to $\mc P^\alpha_T$, we deduce that it also holds almost surely with $Z_0^\gamma$ is chosen according to $M^\alpha$. 
\end{proof}

\begin{remark}
	The exponential time $T$ in the proof above can be replace with a deterministic time $t$ provided we know the existence of a continuous version of the transition density, for which $p_t(x,y) > 0$ for all $x,y\in \mc D$ and all $t > 0$. This is known in the case of a torus \cite{maillard2014liouville}, and similar arguments probably work in the planar case as well. We have made no attempt to check this, however. 
\end{remark}

We now restate and prove \cref{cor:differentiability}:

\begin{corollary}
	Let $\gamma \in (\sqrt 2, 2)$. Then the $\gamma$-Liouville Brownian motion $Z^\gamma$ is Lebesgue-almost everywhere differentiable with derivative zero, almost surely. 
	\label{cor:differentiability.restated}
\end{corollary}

\begin{proof}
	By taking $\alpha = \gamma$ in \cref{thm:regularity.restated}, we know that for $\mu_\gamma$-almost every $t \geq 0$, the change of time $F_\gamma$ has the following growth rate:
	\begin{equation*}
		\lim_{r \rightarrow 0}\frac{\log |F_\gamma(t) - F_\gamma(t+r)|}{\log |r|} = 1 - \frac{\gamma^2}{4}.
	\end{equation*}
	Now, let $\delta \in (0, \frac{1}{2 - \frac{\gamma^2}{2}} - 1)$. We can apply \cref{lem:lbm.scale}, or specifically \cref{eq:Z.bound} in the proof of \cref{lem:lbm.scale}, to see that for $\mu_\gamma$-almost every $t \geq 0$ we have
	\begin{equation*}
		|Z^\gamma_{F_\gamma(t)} - Z^\gamma_{F_\gamma(t) + r}| \leq |r|^{ 1/(2 - \frac{\gamma^2}{2}) - \delta},
	\end{equation*}
	for all $r$ with $|r|$ small enough. But, by the definition of $\mu_\gamma$, the $F_\gamma$ image of a set with full $\mu_\gamma$ measure has full Lebesgue measure. Therefore, we can see that for Lebesgue-almost every $t \geq 0$ we have
	\begin{equation*}
		|Z^\gamma_{t} - Z^\gamma_{t + r}|	\leq |r|^{1/( 2 - \frac{\gamma^2}{2}) - \delta}
	\end{equation*}
	for all $r$ with $|r|$ small enough. Therefore, we have
	\begin{equation*}
		\lim_{r\rightarrow 0} \frac{|Z^\gamma_{t} - Z^\gamma_{t + r}|}{|r|} \leq \lim_{r\rightarrow 0} |r|^{1/( 2 - \frac{\gamma^2}{2}) - \delta - 1} = 0
	\end{equation*}
	where the final inequality is because we have chosen $\delta$ to ensure that $\frac{1}{2 - \frac{\gamma^2}{2}} - \delta - 1 > 0$. So, we certainly have differentiability for $Z^\gamma$, for Lebesgue-almost every $t \geq 0$, and the derivative is equal to zero. 
\end{proof}

\noindent\textbf{Acknowledgements.} I give many thanks to Nathana\"el Berestycki, who first pointed this problem out to me, and has helped me with many of the arguments throughout this paper. 

\bibliography{bib}
\bibliographystyle{plain}
\end{document}